\documentclass{amsart}
\usepackage{amsmath}
\usepackage[dvipsnames]{xcolor}
\usepackage{amsthm}
\usepackage{amsfonts}
\usepackage{graphicx}
\usepackage{stackrel}
\usepackage{caption}
\usepackage{subcaption}
\usepackage{relsize}
\usepackage{adjustbox}
\usepackage{tabularx}
\usepackage{amsfonts}
\usepackage[style=alphabetic, backend = bibtex]{biblatex}
\usepackage{tikz-cd}
\usepackage[shortlabels]{enumitem}
\usepackage[hyperindex,colorlinks,breaklinks]{hyperref}

\newenvironment{absolutelynopagebreak}
{\par\nobreak\vfil\penalty0\vfilneg
	\vtop\bgroup}
{\par\xdef\tpd{\the\prevdepth}\egroup
	\prevdepth=\tpd}

\newtheorem{theorem}{Theorem}
\numberwithin{theorem}{section}
\numberwithin{equation}{section}

\newtheorem{lemma}[theorem]{Lemma}
\newtheorem{example}[theorem]{Example}

\newcommand{\D}{\mathcal{D}}
\newcommand{\p}{\partial}
\newcommand{\G}{\Gamma}
\newcommand{\Dx}{\mathcal{D}_X}
\newcommand{\Dl}{\mathcal{D}_\lambda}
\newcommand{\Ol}{\mathcal{O}_\lambda}
\newcommand{\Oll}{\mathcal{O}_{\lambda-1}}
\newcommand{\Tx}{\Theta_X}
\newcommand{\Cx}{\Omega^1_X}
\newcommand{\Ox}{{\mathcal{O}_X}}
\newcommand{\Oc}{{\mathcal{O}}}
\newcommand{\M}{\mathcal{M}}
\newcommand{\N}{\mathcal{N}}
\newcommand*{\End}[1][\Ox]{\mathcal{E}nd_{\mathbb{C}}(#1)}
\newcommand{\Hom}{\mathcal{H}om}
\newcommand{\gr}{\operatorname{gr}^F}

\newcommand{\DR}{\operatorname{DR}}
\newcommand{\GL}{\operatorname{GL}}
\newcommand{\slt}{\mathfrak{sl}_2}
\newcommand{\Ll}{\mathcal{L}}
\newcommand{\E}{\mathcal{E}}
\newcommand{\ins}{{U_0 \cap U_\infty}}

\newcolumntype{C}[1]{>{\centering\arraybackslash}p{#1}}

\newcommand{\g}{\mathfrak{g}}
\newcommand{\h}{\mathfrak{h}}

\newcommand{\bigstar}[1][0pt]{
	\mathrel{\raisebox{#1}{\text{\Huge $*$}}}
}

\title[Lie algebra representations from regular holonomic $\mathcal{D}$-modules]{An investigation into Lie algebra representations obtained from regular holonomic $\mathcal{D}$-modules}
\date{\today}
\author{Julian Wykowski}
\author{Travis Schedler}

\definecolor{imperiallight}{RGB}{24,142,179}
\definecolor{grn}{RGB}{68, 179, 24}
\hypersetup{
    breaklinks=true,
	linkcolor={grn},
	citecolor={grn},
	urlcolor={grn}
}

\usepackage{breakurl}

\begin{document}
	\maketitle
	\begin{abstract}
		Beilinson--Bernstein localisation \cite{BBL} relates representations of a Lie algebra $\g$ to certain $\D$-modules on the flag variety of $\g$. In \cite{romanov}, examples of $\slt$-representations which correspond to $\D$-modules on $\mathbb{CP}^
		1$ were computed. In this expository article, we give a topological description of these and extended examples via the Riemann--Hilbert correspondence. We generalise this to a full characterisation of $\slt$-representations which correspond to holonomic $\D$-modules on $\mathbb{CP}^1$ with at most 2 regular singularities. We construct further examples with more singularities and develop a computer program for the computation of this correspondence in more general cases.
	\end{abstract}
	\section{Introduction}
	The notion of $\D$-modules has roots in the algebraic study of solutions to linear partial differential equations on a manifold, but has ever since risen to prominence throughout algebraic disciplines, often bridging otherwise distinct fields. The idea originates\footnote{$\D$-modules were first discovered by Mikio Sato, in his work on algebraic analysis. We refer the reader to, e.g., \cite{hotta} for more on this topic, and for a more detailed treatment of anything mentioned in this terse introduction. There are many more concise introductions, e.g.,  \cite{elliott}.}
	in the study of solutions to linear partial differential equations. Indeed, given an open subset $X \subseteq \mathbb{C}^n$ with coordinate ring $\Oc$, an equation of the form $Py = 0$---where $P$ is an element of the $\mathbb{C}$-algebra $D$ of algebraic differential operators with coefficients in $\Ox$---has an associated $D$-module $M = D / DP$. Note that $\mathcal{O}$ is also a $D$-module. We observe that \begin{align*}
		\operatorname{Hom}_D(D/DP, \Oc) &\cong \{\varphi \in \operatorname{Hom}_D(D, \Oc) : \varphi(P) = 0) \\
		&\cong \{f \in \Oc : P(f) = 0\}\}
	\end{align*}
	so that $\operatorname{Hom}_D(M, \Oc)$ describes the vector space of solutions to $Px = 0$ in $X$. This allows us to transport questions in the study of differential equations to a rather algebraic setting---and as we shall see, even the realms of topology or the representation-theory of Lie algebras.
	\subsection*{Algebraic $\D$-modules}
	Formally, fix a smooth algebraic variety $X$ over $k = \mathbb{C}$. Denote by $\Ox$ the structure sheaf on $X$ and by $\End$ the sheaf of $\mathbb{C}$-linear endomorphisms of $\Ox$ (a sheaf of $\mathbb{C}$-algebras), into which $\Ox$ embeds as $f \mapsto f \cdot -$ where $\cdot$ stands for point-wise multiplication. We then define the tangent sheaf $\Tx$ on $X$ as the subsheaf of $\End$ consisting of those elements that satisfy the Leibniz rule: $$\Tx = \{ \xi \in \End : \xi(fg) = \xi(f)g + f\xi(g)\}$$ and its dual, the cotangent sheaf $\Cx = \Hom_\Ox(\Tx, \Ox)$ on $X$, which we shall consider as a complex vector space. Its exterior powers form the de Rham complex $\Omega^\bullet$:
	\begin{equation}\label{complex}
		0 \xrightarrow{\,} \Omega^0(X) \xrightarrow{d} \Omega^1(X) \xrightarrow{d} \Omega^2(X) \xrightarrow{d} \Omega^3(X) \xrightarrow{d} \ldots 
	\end{equation} where $d$ denotes the exterior derivative and we take $\Omega^0(X) = C^\infty(X, \mathbb{C})$. Now, the \textit{sheaf of differential operators} $\Dx$ is the subsheaf of $\mathcal{E}nd_{\mathbb{C}}(\Ox)$ generated by $\Ox$ and $\Tx$, and a (left/right) \textit{$\D$-module} $\M$ on $X$ is simply a sheaf of (left/right) $\Dx$-modules. We will denote the category of left $\Dx$-modules as $\operatorname{Mod}(\Dx)$, and dually, the category of right $\Dx$-modules as $\operatorname{Mod}(\Dx)^{\operatorname{op}}$. Their (bounded) derived categories will be denoted as $D^b(\Dx)$ and $D^b(\Dx^{\operatorname{op}})$, respectively.
	
	Note that every point $x \in X$ has an open affine neighbourhood $U \subseteq X$ with local coordinates $x_1, \ldots, x_n \in \G(U, \Ox)$ such that $\G(U, \Tx)$ is a free $\G(U, \Ox)$-module with basis $\{\partial_1, \ldots, \partial_n\}$ where $[\p_i, \p_j] = 0$ and $[\p_i, x_j] = \delta_{ij}$. In particular, we have:
	\begin{example} \label{affine}
		On the affine space $X = \mathbb{A}^n$, $$\G(X,\Dx) = \mathbb{C}[x_1, \ldots, x_n ; \p_1, \ldots, \p_n]$$ is the Weyl algebra on $n$ variables. 
	\end{example}
	This description of the affine case suggests a natural (exhaustive increasing) filtration on $\Dx$ by order of differential operators. Indeed, in local coordinates $\{x_i\p_i\}$ on affine open $U \subseteq X$ (as described above), we have $$\Gamma(U, \Dx) = \bigoplus_{\alpha \in \mathbb{N}^n} \Oc_U \partial^\alpha$$ where $\partial^\alpha = \partial_1^{\alpha_1} \ldots \partial_n^{\alpha_n}$ in multi-index notation. We can extrapolate to a global filtration $F$ on $\Dx$ by writing $$F_n\Dx = \{P \in \Dx : P|_U \in \bigoplus_{|\alpha| \leq n} \Oc_U \p^\alpha \quad \forall \, U \subseteq X \text{ open affine}\}$$ for all $n \in \mathbb{N}$. The associated graded sheaf of rings $$\gr(\Dx) = \bigoplus_{n \in \mathbb{N}^+}  \frac{F_n\Dx}{F_{n-1}\Dx}$$ is locally $\mathcal{O}_X$-linearly spanned  by (homogeneous) elements of the form $\partial^\alpha$, which we may also regard as local functions on the total space of the cotangent bundle.
In fact this comes from a canonical isomorphism $\gr(\Dx) \cong \pi_*(\Oc_{T^*X})$ where ${\pi : T^*X \to X}$ denotes the projection. The inverse is defined by viewing a fibrewise homogeneous local function on $T^*X \to X$ as a polynomial in vector fields, hence defining a differential operator, uniquely determined modulo differential operators of lower degree.
	
	It can be shown that a $\Dx$-module $\M$ is $\Dx$-coherent\footnote{Given a sheaf of locally Noetherian rings $\mathcal{F}$ on a space $X$, we say that an $\mathcal{F}$-module $\mathcal{M}$ is \textit{quasi-coherent} over $\mathcal{F}$ if any $x \in X$ lies in an open neighbourhood $U$ with exact sequence $$\Gamma(U,\mathcal{F}^{\oplus I}) \to \Gamma(U,\mathcal{F}^{\oplus J}) \to \Gamma(U,\mathcal{M}) \to 0$$ for some index sets $I$ and $J$. If, moreover, $I$ and $J$ can be chosen finite then $\mathcal{M}$ is $\mathcal{F}$-\textit{coherent}.} if and only if it admits an (exhaustive increasing) filtration $F_\bullet$ which agrees with that on $\Dx$, i.e. such that $$(F_n\Dx)(F_m \M) \subseteq F_{n+m}\M$$ holds. Let $\M$ be such a $\Dx$-coherent $\Dx$-module. The associated sheaf of graded modules $\gr(\M)$ is then a module over the sheaf of rings $\gr(\Dx) \cong \pi_*(\Oc_{T^*X})$ which allows us to turn $\gr(\M)$ into an $\Oc_{T^*X}$-module via $$\overline{\gr\M} = \Oc_{T^*X} \otimes_{\pi^{-1}\pi_*\Oc_{T^*X}} \pi^{-1}(\gr \M)$$
	whose support cuts out a sub-variety of $T^*X$ which we call the \textit{characteristic variety} $\operatorname{Ch}(\M)$ of $\M$, and which can be shown independent of choice of compatible filtration. It is a result of Bernstein (cf. \cite{Bernstein}) that $\dim(\operatorname{Ch}(\M)) \geq \dim(X)$. We call $\M$ \textit{holonomic} if equality holds and write $\operatorname{Mod}_h(\Dx)$ for the category of holonomic $\Dx$-modules, as well as $D^b_h(\Dx)$ for its bounded derived category. 
	
	Morally, the characteristic variety describes the solution space of principal symbols in $T^*X$. In case of the motivating example $\M = \Dx / I $ for $I \triangleleft \Dx$ some ideal (corresponding to a system of differential equations), loosely speaking, the size of the solution space of $I$ decreases as the size to the size of $I$ increases, and hence, as the dimension of $\operatorname{Ch(M)}$ decreases. We can think of $\M$ being holonomic roughly when its solution space is as small as possible, i.e. $I$ is a maximally over-determined system of differential equations.
	
	Holonomic can be seen as a step in between $\Dx$-coherence and $\Ox$-coherence. Indeed, if $\M$ is an $\Ox$-coherent $\Dx$-module, then the filtration $F_n\M = \M$ must agree with that on $\Dx$ and we find that locally, $\operatorname{gr}^F \M \cong \M \cong \Ox^r$. Moreover, $\Tx$ then annihilates $\operatorname{gr}^F \M$ as a $\pi_* \Oc_{T^*X}$-sheaf, so indeed $\operatorname{Ch}(\M) = T^*_XX$ is the zero section, whose dimension is $\dim(X)$. Thus any $\Ox$-coherent $\Dx$-module is automatically holonomic.
	
	It can be shown that submodules and quotients of holonomic modules are themselves holonomic. Furthermore, the category of holonomic modules comes equipped with a duality functor called the \textit{Verdier} duality; we skip its definition here and refer the curious reader to Chapter 4.5 of \cite{hotta} instead. One notable consequence of this duality is that a holonomic $\D_X$-module is Artinian if and only if its dual is Noetherian. In the case that $X$ is a quasi-projective variety, any coherent $\D_X$-module is Noetherian (as so is $\Oc_{T^*X}$).
	
	\begin{lemma}\label{Artinian}
		The category of holonomic $\Dx$-modules is Artinian.
	\end{lemma}
	
	A morphism of smooth algebraic varieties $f: X \to Y$ induces the sheaf-theoretic push-forward $f_* : \operatorname{Mod}(\Dx) \to \operatorname{Mod}(\D_Y)$ and pull-back $f^{-1} : \operatorname{Mod}(\D_Y) \to \operatorname{Mod}(\Dx)$. The same can be done on the level of $\D$-modules. We define the $\D$-module \textit{inverse image} $$f^- : \operatorname{Mod}(\D_Y) \to \operatorname{Mod}(\Dx)$$ as the $\Ox$-module $$f^-(\M) = \Ox \otimes_{f^{-1}\Oc_Y} f^{-1}\M$$ whereon $\xi \in \Tx$ acts via the imposed Leibniz rule $$\xi(f\otimes s) = \xi(f) \otimes s + f\cdot\tilde{\xi}(s)$$ and $\tilde{\xi}$ denotes the image of $\xi$ under the map $\Tx \to \Ox \otimes_{f^{-1}\Oc_Y}\Theta_Y$ of $\Ox$-modules which is $\Ox$-dual to $\Ox \otimes_{f^{-1}\Oc_Y} f^{-1} \Cx \to \Cx$. On the other hand, the $\Dx$-module \textit{direct image} is more easily written down in terms of right $\Dx$-modules; we define:
	$$f_+ :  \operatorname{Mod}(\Dx)^{\operatorname{op}} \to \operatorname{Mod}(\D_Y)^{\operatorname{op}}$$ as $$f_+(\M) = f_*(\M \otimes_{\Dx} \D_{X\to Y})$$ where $\D_{X\to Y} = f^-(\D_Y)$ is the \textit{transfer module}, which is a $(\Dx,f^{-1}\D_Y)-$bimodule. Dually, we can write down the opposite transfer module $$\D_{Y \leftarrow X} = \Omega_X \otimes_{\Ox} \D_{X\to Y} \otimes_{f^{-1}\Oc_Y} f^{-1} \Omega_Y^\vee$$ where $(-)^\vee = \mathcal{H}om_{\Ox}(-,\Ox)$ denotes the sheaf-theoretic dual, and then $$f_+ :  \operatorname{Mod}(\Dx) \to \operatorname{Mod}(\D_Y)$$ is given by $$f_+(\M) = f_*(\D_{Y \leftarrow X} \otimes_{\Dx} \M)$$ with inherited $\D_Y$-module structure. The inverse image functor is right exact; consequently, we may pass over to the realm of (bounded) derived categories by defining the left derived functor $$\mathbb{L}f^- : D^b(\D_Y) \to D^b(\D_X)$$ and its shifted analogue $$f^\dagger = \mathbb{L}f^-[\dim X - \dim Y] : D^b(\D_Y) \to D^b(\D_X)$$
	which plays an important role in that it is in certain contexts an inverse to the derived direct image
	\begin{align*}
		\int_f: D^b(\Dx) &\to D^b(\D_Y) \\
		\M^\bullet &\mapsto \mathbb{R}f_+(\D_{Y \leftarrow X} \otimes^{\mathbb{L}}_{\D_X} \mathcal{M}^\bullet).
	\end{align*}
	The reader averse to technicalities shall be reassured that a brief category-theoretic description of the functors $f^\dagger$ and $\int_f$ will suffice to formulate and prove the results in this paper. We summarise the necessary in the following result and refer to Chapter 2.7 in \cite{hotta} for proofs.
	\begin{lemma} \label{adjoint} If $f : X \hookrightarrow Y$ is a locally closed embedding of smooth varieties, then
		\begin{enumerate}[(a)]
		\setlength\itemsep{0.1em}
			\item the functors $f^\dagger$ and $\int_f$ preserve holonomicity;
			\item for any $\M \in D^b(X)$, we have $f^\dagger \int_f (\M) = \M$;
			\item $f^\dagger$ has a left adjoint $\int_{f!} : D^b(X) \to D^b(Y)$;
			\item if $f$ is proper then $\int_f \cong \int_{f!}$. 
		\end{enumerate}
	\end{lemma}
	
	Since $\Ox \leq \Dx$, any $\Dx$-module $\M$ is simultaneously an $\Ox$-module; the $\Dx$-module structure on $\M$ extends the $\Ox$-module structure only insofar as it specifies an additional action $\Tx \rightarrow \End[\M]$. One can verify that this action is $\mathbb{C}$-linear, commutes with the Lie bracket, satisfies the Leibniz rule and that any action satisfying these three conditions specifies a $\Dx$-module structure on any $\Ox$-module. In particular, the data of an $\Ox$-coherent $\Dx$-module $\M$ is equivalent
	to the data of an $\Ox$-module $\M$ together with an (algebraic) connection $$\nabla : \M \to \Cx \otimes_\Ox \M$$ which is flat (i.e. integrable). In the case that $X \cong \mathbb{C}$, any (flat) connection is trivialisable. Choosing such a trivialisation amounts to finding an isomorphism $\nabla \cong d + A(z)$ where $A(z)$ is a matrix of meromorphic one-forms. If $A(z)$ can be chosen in a way that all its entries have poles of order at most 1, then we say that $\nabla$ has \textit{regular singularities}. Analytically, this corresponds to the differential equation $\frac{dF}{dz} = AF$ having moderate growth at its singularities; algebraically, it means that the operator $\p^z$ does not decrease the degree of a polynomial by more than one, or alternatively, that $\M$ is stable under the action of the operator $z\p^z$.
	
	This algebraic characterisation allows us to extrapolate the definition for more general $\D$-modules. Assume first that $C$ is a smooth curve and $\iota : C \hookrightarrow \hat{C}$ is a projectification. Let $\D_{\hat{C}}^C$ be the subsheaf of $D_{\hat{C}}$ locally generated by $\Oc_{\hat{C}}$ and the vector fields which vanish outside $C$. We say that an $\Oc_C$-coherent $\D_C$-module $\M$ has \textit{regular singularities} if there is an $\Oc_{\hat{C}}$-coherent $\Ox$-submodule $\mathcal{N}$ of $\iota_+(\M)$ which is stable under the action of $\D_{\hat{C}}^C$. This concept is invariant under taking extensions and subquotients.
	
	Finally, in the case where $X$ is any variety, we say that a simple holonomic $\Dx$-module $\M$ has \textit{regular singularities} if there is an open dense subset $U \subseteq X$ such that $\mathcal{N} = \M|_{U}$ is $\Oc_U$-coherent and any restriction of $\mathcal{N}$ to a smooth curve has regular singularities in the sense of the previous definition. A holonomic $\Dx$-module has regular singularities if so do all its simple subquotients. 
	
	\subsection*{Riemann--Hilbert Correspondence}
	The inceptive result utilising $\D$-modules as a bridge between mathematical fields is perhaps the Riemann--Hilbert correspondence (RH). This result was shown first by Deligne \cite{deligne}, formulated as an equivalence between the category of flat connections with regular singularities on algebraic vector bundles on a smooth complex variety $X$ and the category of locally constant sheaves of complex vector spaces on $X$. Later, it was generalised to a correspondence between the derived category of holonomic $\Dx$-modules with regular singularities and the derived category of coherent sheaves of complex vector spaces on $X$, independently by Kashiwara \cite{kashiwara} and Mebkhout \cite{Mebkhout}. We introduce it here in this latter form.
	
	A sheaf of complex vector spaces $\mathcal{F}$ on a variety $X$ is \textit{constructible} if there exists a stratification\footnote{\label{stratification}For our purposes, a \textit{stratification} of a smooth variety $X$ shall be a finite decomposition into connected locally closed subsets $X = \coprod X_i$ called \textit{strata}, such that the closure of each stratum is a union of strata.} of $X$ such that $\mathcal{F}$ is locally constant of finite rank on each $X$. Let $\operatorname{Sh}_c(X; \mathbb{C})$ be the category of constructible sheaves of finite-dimensional complex vector spaces and denote by $D^b(\operatorname{Sh}_c(X;\mathbb{C}))$ the associated bounded derived category. Similarly, let $D^b_{rh}(\Dx)$ be the bounded derived category of holonomic $\Dx$-modules with regular singularities. We now want to define a certain “solutions functor" which associates to a $\D$-module describing a differential equation the sheaf of its solutions. Since algebraic differential equations need not  have algebraic solutions, we should rather carry out this construction in an analytic setting.
	
	A smooth variety $X \subseteq \mathbb{P}^n$ comes with a natural complex manifold $X^{\operatorname{an}}$ given by the complex structure on $ X(\mathbb{C})$. As the Zariski topology is coarser than the Euclidean, this yields a continuous map $\iota : X^{\operatorname{an}} \to X$, and hence, a functor \begin{align*}
		(-)^{\operatorname{an}} : \operatorname{Mod}(\Dx) &\to \operatorname{Mod}(\D_{X^{\operatorname{an}}}) \\
		\M &\mapsto \Oc_{X^{\operatorname{an}}} \otimes_{\iota^{-1}\Ox} \iota^{-1} \M
	\end{align*}
which is exact and induces a functor on the associated derived categories $$(-)^{\operatorname{an}} : D^b(\operatorname{Mod}(\Dx)) \to D^b(\operatorname{Mod}(\D_{X^{\operatorname{an}}})).$$
This allows for the definition of a generalised "solutions functor", and therewith, the formulation of the Riemann--Hilbert correspondence:

\begin{theorem}[Riemann--Hilbert Correspondence]
	The de Rham functor 
	\begin{align*}
		\DR : D^b_{rh}(\Dx) &\to D^b(\operatorname{Sh}_c(X;\mathbb{C})) \\
		\M^\bullet &\mapsto \Omega^\bullet_{X^{\operatorname{an}}} \otimes^{\mathbb{L}}_{\D_X^{\operatorname{an}}} (\M^\bullet)^{\operatorname{an}}
	\end{align*}
is an equivalence.
\end{theorem}
Given the natural embedding of categories $$\operatorname{Mod}_{rh}(\Dx) \hookrightarrow D^b_{rh}(\Dx)$$ one may then ask what subcategory of $D^b(\operatorname{Sh}_c(X))$ corresponds to the category of $\Dx$-modules $\operatorname{Mod}(\Dx)$ under the de Rham functor. The answer turns out to be the category of \textit{perverse sheaves} (cf. \cite{perverse_sheaves}), which is a subcategory of the derived category of sheaves that itself behaves like a category of sheaves in certain ways; we refer the reader to Chapter 8 of \cite{hotta} for a precise treatment of the subject.

\subsection*{Beilinson--Bernstein Localisation}
	We now introduce the second main correspondence theorem we will be using: Beilinson--Bernstein localisation. Let $G$ be a semisimple algebraic group over $\mathbb{C}$, $T$ a maximal torus of $G$ and $B \supseteq T$ a Borel subgroup of $G$. Write $\g = \operatorname{Lie}(G)$, $\h = \operatorname{Lie}(T)$ and $\mathfrak{b} = \operatorname{Lie}(B)$ for the associated Lie algebras. All Borel subgroups are conjugate in $G$, yielding a variety $X = G/B$ equipped with a Lie group action $$G \longrightarrow \operatorname{Diff}(X)$$ given by conjugation. The infinitesimal of this action is a Lie algebra morphism $$\g \longrightarrow \G(X,\Theta(X))$$ which induces, by the universal property of the universal enveloping algebra, a map $$\Phi : U(\mathfrak{g}) \longrightarrow \G(X,\Dx).$$ Beilinson and Bernstein showed in \cite{BBL} that this map is surjective and that its kernel is $\ker(\chi_0) \cdot U(\g)$, where $\chi_0 : Z(\g) \to \mathbb{C}$ is the trivial central character. They went further to generalise this to the case of twisted central characters, which allow for a complete description of all $\g$-representations. Indeed, given any $\g$-representation, the centre $Z(\g)$ must act as scalars by Schur's Lemma, and this “central character" factors as $$\chi : Z(\g) \xrightarrow{\,\, \gamma\,\, }U(\h) \xrightarrow{\,\,\lambda\,\,} \mathbb{C}$$ where $\gamma$ denotes the Harish-Chandra homomorphism (cf. \cite{Hotta1998EquivariantD} and $\lambda \in \h^*$ is some dominant and regular weight. We associate to $\lambda$ a one-dimensional $B$-representation $\mathbb{C}_\lambda$, whereon $B$ acts via $\lambda \circ \pi$ for $\pi : B \to T$ the projection; this in turn corresponds to a $G$-equivariant line bundle $$\pi_\lambda : \frac{G \times \mathbb{C}_\lambda}{B} \longrightarrow G/B = X$$ where $B$ acts as $b \cdot (g,v) = (gb^{-1}, b\cdot v)$ and we denote by $\mathcal{L}(\lambda)$ the sheaf of algebraic sections of $\pi_\lambda$ with $\Ox$-module structure given by point-wise multiplication. In this way, we may incorporate the twist by $\lambda$ into $\Dx$, yielding a \textit{sheaf of twisted differential operators} \begin{equation}\label{TDO}
		\Dl = \bigcup_{n \in \mathbb{N}} F_n\Dl
	\end{equation} where $F_0\Dl = \Ox$ and $$F_n\Dl = \{\xi \in \End[\mathcal{L}(\lambda + \rho)] : \forall f \in \Ox, [\xi, f] \in F_{n-1}\Dl\}.$$
	Using the same argument as in the untwisted case, we obtain a map $$\Phi_\lambda : U(\g) \longrightarrow \G(X, \Dl)$$ which can be written down explicitly using the $G$-invariant structure on $\mathcal{L}(\lambda + \rho)$. In the case where $G$ is a matric Lie algebra, it amounts to 
	\begin{equation}\label{diffexp}
		\Phi_\lambda(a)(s)(x) = \left. \frac{\p}{\p t} e^{ta} \cdot s(e^{-ta}x) \right|_{t=0}
	\end{equation} for $a \in U(\g), s \in \Lambda(\lambda + \rho)$ and $x \in X$. The key result in \cite{BBL} is that the map $\Phi_\lambda$ is surjective and has kernel $\ker(\chi_\lambda)\cdot U(\g)$. Now the global sections of any $\Dl$-module $\M$ can be given the structure of a $\g$-representation with central character $\chi_\lambda$. Beilinson--Bernstein localisation gives a strong converse result; we state it here only for regular characters:
	\begin{theorem}[Beilinson--Bernstein Localisation]\label{BBL}
		For any dominant and regular weight $\lambda \in \h^*$, there exists an equivalence of categories:
		\[
		\begin{tikzcd}
			\operatorname{Mod}_{qc}(\Dl) \arrow[rr, swap, bend right = 15, "{\Gamma(X,-)}"] && \operatorname{Mod}(\g, \chi_\lambda) \arrow[swap, ll,bend right = 15, "\Dl \otimes_{U\g} -"]
		\end{tikzcd}
	\]
	where $\operatorname{Mod}_{qc}(\Dl)$ is the category of quasi-coherent $\Dl$-modules and $\operatorname{Mod}(\g, \chi_\lambda)$ is the category of $\g$-representations whose centre acts as $\chi_\lambda$. 
	\end{theorem}

	\section{A dictionary of $\D$-modules}
	The key philosophy behind this paper is to outline a dictionary between the topological, geometric and representation theoretic aspects of $\D$-modules, with focus on the role of singularities. The significance for the examination of singularities becomes evident when one tries to connect the two correspondences at hand. Recall that given an algebraic group $G$ with Borel subgroup $B$, Beilinson--Bernstein localisation (Theorem \ref{BBL}) provides an equivalence of the category of (potentially twisted) quasi-coherent $\D$-modules on $X = G/B$, and the category of $\g$-representations. In turn, the Riemann--Hilbert correspondence provides an equivalence of a subcategory of the (bounded derived) category of $\D$-modules on $X$ (namely, the subcategory of regular holonomic modules) to the (bounded derived) category of constructible sheaves on $X$. To connect these two correspondences, one thus only has to find the image of the functor $$\G(X,-) \circ \DR(-): D^b(\operatorname{Sh}_c(X)) \to \operatorname{Mod}(\g, \chi_0)$$ i.e. the subcategory of $\g$-representations with trivial central character that corresponds to regular holonomic $\Dx$-modules. We summarise the situation in the following table:
	\renewcommand{\arraystretch}{2}
		\begin{center}
			\begin{tabular}{C{0.28\linewidth}|C{0.34\linewidth}|C{0.28\linewidth}}
				\textbf{Topology} & \textbf{Geometry} & \textbf{Rep. Theory} \\ \hline
				– & $\D$-modules on $X = G/B$ & Representations of $\g = \operatorname{Lie}(G)$ \\
				Perverse sheaves on $X = G/B$ & Regular holonomic $\D$-modules on $X = G/B$ & \textbf{??}
			\end{tabular}
		\end{center}
	\renewcommand{\arraystretch}{1}
	One task facilitated by this dictionary is the classification of certain Lie algebra representations. A classification of \textit{all} $\g$-representations is utterly hopeless even for the archetypal example of $\g = \slt(\mathbb{C})$, since---as we are about to discuss---this category contains all representations of free groups on finitely many elements. However, this obstruction hints at the same time at a possible structure we may endow the mess with, namely, through understanding subcategories of representations that correspond to $\D$-modules with (few) specified singularities.
	
	Certain subcategories of naturally occurring $\g$-representations have already been classified: classically, given a semi-simple complex Lie algebra $\g$, the finite-dimensional $\g$-representations are classified using the highest weight theorem; a broader collection of Lie algebras called “category $\Oc$" (see \cite{CatO} for a lightning introduction) can be classified using Verma modules. It turns out that these subcategories correspond to $\D$-modules with at most 2 regular singularities, hinting at a potential generalisation of this correspondence to all representations which correspond to holonomic $\D$-modules with specified singularities.
	
	Hereinafter, we will say that a $\g$-representation has certain (regular or irregular) \textit{singularities} at points on its flag variety if it corresponds to a $D$-module with those singularities under the Beilinson--Bernstein correspondence.
	
	\subsection*{$\Oc$-coherent $\D$-modules on smooth curves}
	Let $X$ be an irreducible smooth curve and $\iota : X \hookrightarrow C$ an embedding into a smooth projective curve $C$. Let $\M$ an $\Oc$-coherent $\D$-module on $X$, i.e. an $\Oc$-module together with flat connection $\nabla$, and assume that $\nabla$ has regular singularities on $C \backslash X$. Under the Riemann--Hilbert correspondence, $\M$ yields the complex $\DR(\M) = \Omega_{X^{\operatorname{an}}}^\bullet \otimes^{\mathbb{L}}_{\Dx^{an}} \M^{\operatorname{an}}$ of analytic $\D$-modules. Using the resolution in (\ref{complex}), we find that this complex has trivial homology except in the leftmost place, where the differential agrees with $\nabla$. There, we have \begin{equation}\label{0th homology}
			H^0\DR(\M) = H^0(\Omega_{X^{\operatorname{an}}}^\bullet \otimes^{\mathbb{L}}_{\Dx^{an}} \M^{\operatorname{an}}) \cong \ker(\nabla)
	\end{equation} and since $\nabla$ is $\mathbb{C}$-linear, $\Ll = H^0\DR(\M)$ carries the structure of a sheaf of complex vector spaces on $X$. The Frobenius theorem of differential topology says that $\Ll$ is in fact a constructible sheaf on $X$, i.e. locally constant of finite rank on the strata of some stratification (see footnote \ref{stratification}). By the assumption of $\Oc$-coherence, this stratification actually consists only of a the single stratum $X$.
	
	By path-connectedness, we may join any two points $x,y \in X$ by a path $\gamma : [0,1] \to X$, which induces an isomorphism of stalks $\Ll_{\gamma(0)} \xrightarrow{\,\,\sim\,\,} \Ll_{\gamma(1)}$,\footnote{As $X$ is irreducible, it is also a connected complex manifold, so it must be path-connected under the analytic topology. Thus, given two points $x,y \in X$, we may find a path $\gamma : [0,1] \to X$ from $x$ to $y$; any local system on $[0,1]$ must be constant and hence $\gamma^{-1}\Ll$ must be constant as well. It follows that $\Ll_x \cong (\gamma^{-1}\Ll)_0 \cong (\gamma^{-1}\Ll)_1 \cong \Ll_y$.} so $\Ll_U$ has isomorphic stalks $L = \Ll_x$ at each point. Note that the freedom of these isomorphisms to differ from the identity is precisely the freedom of locally constant sheaves to differ from the constant sheaf on $X$.  It is not hard to show that homotopic paths induce the same isomorphism and that composition of isomorphisms commutes with the composition of maps.  Thus we obtain a group homomorphism $$\rho_\Ll : \pi_1(X) \longrightarrow \GL(L)$$
	which we call the \textit{monodromy representation of $\Ll$}.
	
	Conversely, let $\rho : \pi_1(X) \to \GL(L)$ be a monodromy representation (for $L$ some complex vector space of finite dimension). As $X$ is a connected complex manifold, we may chose a universal cover $p : \tilde{X} \to X$. Endowing $L$ with the discrete topology, $\pi_1(X)$ acts on $L$ via $\rho$ and on $\tilde{X}$ as deck transformations: a homotopy class of paths $[\gamma]$ acts on a point $\tilde{x} \in p^{-1}(x)$ as $[\gamma] \cdot \tilde{x} = \tilde{\gamma}_{\tilde{x}}(1)$ for $\tilde{\gamma}_{\tilde{x}}$ a lift of $\gamma$ starting at $\tilde{x}$. Thus we may form the analytic bundle \begin{equation}\label{bundle}
		\pi_\rho : \frac{\tilde{X} \times L}{\pi_1(X)} \longrightarrow X
	\end{equation} whose sheaf of sections we shall denote as $\Ll_\rho$. Since this construction is independent of local perturbation, $\Ll_\rho$ is locally constant on $X$. Via the Riemann--Hilbert correspondence, $\Ll_\rho$ corresponds to a unique $\Ox$-coherent $\Dx$-module $\M$; explicitly, $\M$ can be realised as the $\Ox$-module
	\begin{equation}\label{inverse RH}
		\M_\rho = \Ox \otimes_{\mathbb{C}_X} \Ll_\rho
	\end{equation}
	whereon $\Dx$ acts as the connection
	\begin{align*}
		\nabla: \M_\rho = \Ox \otimes_{\mathbb{C}_X} \Ll_\rho &\to \Cx \otimes_{\mathbb{C}_X} \Ll_\rho \cong \Cx \otimes_{\Ox} \M_\rho \\
		(f,v)&\mapsto df \otimes v.
	\end{align*}
	We may find the monodromy of $\nabla$ on $C \backslash X$ and glue an analytic connection with (imposed) regular singularities at each of the points in $ C \backslash X$. By Serre's GAGA (cf. \cite{GAGA}), this analytic connection on $C$ entails a unique algebraic structure, giving rise to an algebraic vector bundle $\overline{\Ll}$ on $C$ with algebraic flat connection.
	
	One can see that this construction is inverse to the functor in (\ref{0th homology}), making this an equivalence. We have shown that $\Ox$-coherent $\Dx$-modules with regular singularities on $C$ are equivalent to finite-dimensional $\pi_1(X)$-representations.
	
	\subsection*{Classification of holonomic $\D$-modules}
	Having established a topological characterisation of $\Oc$-coherent $\D$-modules on a smooth curve $X$, we'd like to extend this idea to holonomic $\D$-modules in general. Let $\M$ be a holonomic $\D$-module on a smooth variety $X$. As we argued in the introduction, $\M$ fails to be $\Oc$-coherent only insofar as its characteristic variety $\operatorname{Ch}(\M)$ extends outside the identity section $T^*_XX$. The projection $\pi : \operatorname{Ch}(\M) = \M$ is surjective, and by assumption of holonomicity we have $\dim(\operatorname{Ch}(\M)) = \dim(X)$, so there must be some open dense
	\begin{equation}
		j : U \hookrightarrow X \label{gen}
	\end{equation}
such that $\pi$ has finite fibres on $U$ (here we use Theorem 1.25 in \cite{shafarevich}). On the other hand, the fibre of $\pi$ at a point $x \in X$ whose preimage intersects the complement of $T_X^*X$ must be stable under scalar multiplication, so in particular $\pi^{-1}(x)$ must have positive dimension and $x \notin U$. Thus $\operatorname{Ch}(\M|_U) = T^*_UU$ and $\M|_U$ is $\Oc|_U$-coherent.

We may use certain well-behaved extensions of modules on subsets of the form (\ref{gen}), as well as their closed complements, to establish a classification of holonomic $\D$-modules which builds on that of the preceding subsection. The extension we define works on any locally closed subvariety $A \subseteq X$ with affine embedding $\iota : A \hookrightarrow X$.

Suppose that we are given such an embedding and some holonomic $\D_A$-module $\M$. By Lemma \ref{adjoint}(b), there exists a bijection
	\begin{equation*}
		I : \Hom_{D_h^b(\D_A)}(\M, \M) \xrightarrow{\,\, \sim \, \,} \Hom_{D_h^b(\D_A)} \left(\M , \iota^\dagger \int_\iota \M \right) \label{a}
	\end{equation*}
	and by Lemma \ref{adjoint}(c), there exists a bijection
	\begin{equation*}
		 J : \Hom_{D_h^b(\D_A)} \left(\M , \iota^\dagger \int_\iota \M \right) \xrightarrow{\,\, \sim \, \,} \Hom_{D_h^b(\Dx)} \left(\int_{\iota!} \M , \int_\iota \M \right) \label{b}.
	\end{equation*}
 	As $\iota$ is affine, $\D_{X \leftarrow A}$ is locally free of finite rank
 	and the higher homology groups of  $\int_\iota \M$ and $\int_{\iota!} \M$ vanish; thus we may regard both $\int_\iota \M$ and $\int_{\iota!} \M$ as $\Dx$-modules. This begs the definition of the natural $\Dx$-module $$\E(A, \M) = \operatorname{Im}(I \circ J (\operatorname{id}_\M))$$ which lies in between $\iota_!(\M)$ and $\iota_+(\M)$. We refer to it as the \textit{minimal extension} of $\M$.  It is a consequence\footnote{Although this deduction is far from trivial. See Theorem 3.4.2 in \cite{hotta}.} of Kashiwara's theorem that the minimal extension of a simple $\D_A$-module $\M$ is the unique simple quotient of $\int_{\iota!} \M$, where by a \textit{simple} $\D$-module we mean one that has exactly two distinct submodules, namely itself and the trivial module.
 	
 	It turns out that these minimal extensions are all one needs to build a holonomic $\D$-module from simple holonomic $\D$-modules on a stratification of $X$. Indeed,
 	\begin{theorem}\label{comp series}
 		Any holonomic  $\Dx$-module $\M$ admits a composition series $$0 \trianglelefteq \M_n \trianglelefteq \M_{n-1} \trianglelefteq \ldots \trianglelefteq \M_1 = \M$$ such that for each $i \in \{1,\ldots,n\}$, there is isomorphism $\frac{\M_{i-1}}{\M_i} \cong \E(A_i, \mathcal{N}_i)$ for some locally closed $A_i \subseteq X$ and $\Oc_{A_i}$-coherent $\D_{A_i}$-module $\mathcal{N}_i$. 
 	\end{theorem}
 \begin{proof}
 		Let $\M$ be a holonomic $\Dx$-module. As submodules and quotients of holonomic modules are themselves holonomic, we may choose holonomic submodules $\ldots \subseteq \M_3 \subseteq \M_2 \subseteq \M_1 = \M$ such that the quotients in this sequence are holonomic and simple. By Lemma \ref{Artinian}, this sequence must terminate, yielding a tower $$0 \trianglelefteq \M_n \trianglelefteq \M_{n-1} \trianglelefteq \ldots \trianglelefteq \M_1 = \M$$ whose quotients are simple. Choose any one such quotient $\E = \frac{\M_i}{\M_{i+1}}$. Since $\E$ is holonomic as well, we may find an open dense subset with affine embedding as in (\ref{gen}), such that $j^\dagger \E$ is an (algebraic) flat connection. By Lemma \ref{adjoint}, we have
 		$$\Hom_{\Dx}\left( j^\dagger \E , j^\dagger \E \right) \cong \Hom_{\Dx} \left(\int_{j!} j^\dagger \E , \E \right)$$
 		and by passing $\operatorname{id_\M}$ through this isomorphism, we obtain a map $\int_{j!} j^\dagger \E \to \E$. As $\E$ is simple, this map must be surjective and $\E$ is in fact a simple quotient of $\int_{j!} j^\dagger \E$; as we remark above, any such module must have $\E \cong \E(U, X)$. 
 	\end{proof}
	
	\section{The case $\g = \slt$}
	The remainder of this paper should be treated as an investigation into the role that $\D$-module singularities play in their corresponding $\g$-representations. We will work entirely over $G = \operatorname{SL}_2(\mathbb{C})$; in this section, we briefly introduce this Lie group (and more pertinently, its Lie algebra), alongside with the structure of $\D$-modules on its flag variety. Our analysis is based on \cite{romanov} and we refer the reader thereto for detailed derivations.
	
	Recall that $G = \operatorname{SL}_2(\mathbb{C})$ is the Lie group of complex $2\times2$-matrices with unit determinant and its Lie algebra $\g = \slt(\mathbb{C})$ may be identified with the algebra of traceless $2 \times 2$-matrices over $\mathbb{C}$. The standard basis for $\g$ is given by the matrices
	
	\[
	E = \begin{pmatrix} 0 & 1 \\ 0 & 0 \end{pmatrix}, \qquad 
	F = \begin{pmatrix} 0 & 0 \\ 0 & 1 \end{pmatrix}, \qquad
	H = \begin{pmatrix} 1 & 0 \\ 0 & -1 \end{pmatrix}
	\]
	
	and we note that $\h = \mathbb{C} \cdot H$ is a Cartan subalgebra. The set of upper-triangular matrices in $G$ forms the  Borel subgroup
	$$B = \left\{\begin{pmatrix} a & b \\ 0 & a^{-1} \end{pmatrix} : a \in \mathbb{C}, b \in \mathbb{C} \right\}$$ and its Lie algebra $\mathfrak{b} = \operatorname{Lie}(B)$ is a Borel subalgebra of $\g$ which contains $\h$; indeed, writing $\mathfrak{n}^+$ and $\mathfrak{n}^-$ for the spans of the positive and negative roots spaces of $\g$ w.r.t. $\h$, we obtain $\mathfrak{b} = \h \oplus \mathfrak{n}^+$ and $\g = \mathfrak{b} \oplus \mathfrak{n}^-$. As $\h$ is spanned by $H$, a quick calculation reveals that the only two roots of $\g$ w.r.t. $\h$ are the elements of $\h^*$ generated by $\lambda^+ = H \mapsto 2$ and $\lambda^{-}: H \mapsto -2$, respectively. Hence the Weyl group of w.r.t. $\h$ contains only the identity and one reflection, i.e. $W = \mathbb{Z}/2\mathbb{Z}$, the Weyl element is $\rho = \frac{1}{2} \cdot 2$, and the weight lattice in $\h^*$ is generated by $\lambda^{+}$, i.e. $P = \langle \lambda^+\rangle \cong \mathbb{Z}$. 
	
	Consider now the natural action of $G$ on $\mathbb{C}^2 \backslash \{0\}$ via matrix multiplication, which is transitive and under which the stabiliser of the line $\mathbb{C} \cdot (1,0)$ is precisely $B \subseteq G$. Thus the flag variety of $G$ may be identified as 
	\begin{align*}
		X = G/B &\xrightarrow{\,\, \sim \,\,} \mathbb{CP}^1 \\
		g &\mapsto [g \cdot (1,0)]
	\end{align*}
	and $G$ acts on $X$ via matrix multiplication. We will distinguish the points $0 = [1:0]$ and $\infty = [0:1]$ on $X$; this allows for the definition of an atlas on $X$ given by the charts $U_0 = X \backslash \{\infty\}$ and and $U_\infty = X \backslash \{0\}$ and the coordinates \begin{align*}
		z : U_0 &\to \mathbb{C}&w : U_\infty &\to \mathbb{C}, \\
		[x : y] &\mapsto y/x & [x : y] &\mapsto x/y
	\end{align*}
	which are related on $U_0 \cap U_\infty$ via $z = w^{-1}$, so that $U_0 \cap U_\infty \cong \mathbb{C}^\times$ and the diffeomorphism $x^{-1} : \mathbb{C}^\times \to \mathbb{C}^\times$ is the transition map between the two charts. To avoid confusion, we fix throughout this paper the following notation for all embeddings which might be involved:
		
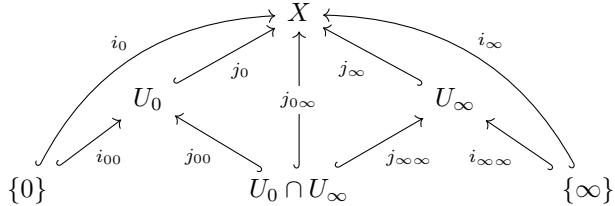
\[
		\begin{tikzcd}
			&&X&& \\
			&U_0 \arrow[ur,"j_0", hookrightarrow,swap]&&U_\infty \arrow[ul,"j_\infty", hookrightarrow]& \\
			\{0\} \arrow[ur,"i_{00}", hookrightarrow,swap]\arrow[uurr,"i_{0}", hookrightarrow, bend left = 30]&&\ins \arrow[uu, "j_{0\infty}" description, hookrightarrow]\arrow[ul,"j_{00}", hookrightarrow] \arrow[ur,"j_{\infty\infty}", swap, hookrightarrow] && \{\infty\} \arrow[ul,"i_{\infty\infty}", hookrightarrow] \arrow[uull,"i_{\infty}", hookrightarrow, bend right = 30,swap]
		\end{tikzcd}\]
	so that $i$'s denote closed immersions and $j$'s denote open immersions. We note that $\G(U_0, \Ox) = \mathbb{C}[z]$ and $\G(U_\infty, \Ox) = \mathbb{C}[w]$. Given a line bundle $\Ll$ on $X$, we thus have isomorphisms $$\varphi_0 : \G(U_0,\Ll) \xrightarrow{\sim} \mathbb{C}[z] \qquad \varphi_\infty : \G(U_\infty, \Ll) \xrightarrow{\sim} \mathbb{C}[w]$$ and restricting to $\ins$, we obtain an automorphism $$\psi = \varphi_0|_\ins  \circ (\varphi|_\ins)^{-1} : \G(\ins, \Ox) \xrightarrow{\sim} \G(\ins, \Ox)$$ which, roughly speaking, describes how $\Ll$ is ``glued together" on the intersection of these two affine charts; this is the only degree of freedom in the choice of $\Ll$. As $\G(\ins, \Ox)$ is generated by unity over itself, any module automorphism must have the form $g \mapsto f\cdot g$ for some $f \in \G(\ins, \Ox)$, so in particular $\psi(g) = f_0 \cdot g$ and $ \psi^{-1}(g) = f_0^{-1} \cdot g$ where $f\cdot f^{-1} = 1$. This means that neither $f_0$ nor its inverse can have a zero in $\ins$ and we must have $f_0 = cz^n$ for some constant $c \in \mathbb{C}^\times$ and $n \in \mathbb{Z}$. It can be shown that $n$ determines $\psi$, and hence the line bundle $\Ll$, up to isomorphism\footnote{See Corollary 6.17 in \cite{hartshorne}} regardless of the choice of $c$ (so we may in fact assume that $c = 1$). We will henceforth denote by $\psi_n$ and $\Oc_n$ the automorphism and line bundle, respectively, associated to $n$. Explicitly,
	\begin{equation}\label{small psi}
		\psi_n : \G(\ins,\Oc_n) \xrightarrow{\sim} \G(\ins, \Oc_n) : g \mapsto z^n \cdot g
	\end{equation} holds.
	\subsection*{Twisted differential operators on $X = \mathbb{CP}^1$}
	Given a dominant integral weight $\lambda \in P^+ \cong \mathbb{Z}_{>0}$, we may construct the sheaf of differential operators on $X$ twisted by $\lambda$ in accordance with the definition in (\ref{TDO}). The line bundle $\mathbb{C}_\lambda$ then corresponds to $\Ol$ and $\Dl$ is the sheaf of twisted differential operators on $\Oc_{\lambda - 1}$, where the translation by $-1$ comes from the Weyl element $\rho = 1$. For integer $\lambda$, we have $\mathcal{D}_\lambda = \operatorname{Diff}(\Oll, \Oll)$ so we obtain:
	
	\begin{lemma}\label{equivtwist}
	    If $\lambda \in \mathbb{Z}$ then the functor $$ \Oll \otimes_{\Ox} - : \operatorname{Mod}(\Dx) \xrightarrow{\,\, \sim \,\,} \operatorname{Mod}(\Dl)$$
	    is an equivalence between the category of $\Dx$-modules $\operatorname{Mod}(\Dx)$ and the category of $\Dl$-modules $\operatorname{Mod}(\Dl)$.
	\end{lemma}
	
	Locally, the sheaf $\Dl$ is isomorphic to $\Dx$, so that \begin{equation}\label{glue}
			\G(U_0, \Dl) = \mathbb{C}[z,\p_z] \qquad \& \qquad \G(U_\infty, \Dl) = \mathbb{C}[w,\p_w]
	\end{equation} as per Example \ref{affine}. Moreover, $\Dl$ is uniquely determined by the unique ring automorphism $$\Psi_\lambda : \G(\ins, \Dl) \to \G(\ins, \Dl)$$ which agrees with $\psi_{\lambda-1}$, in the sense that for any $P \in \Gamma(\ins, \Dl)$, we have $$\psi_\lambda(P) \circ \psi_{\lambda-1} = \psi_{\lambda-1} \circ P$$ (here $P$ is considered as an element of $\G(\ins, \Dl)$ on the left and as a $\G(\ins, \Oll)$-endomorphism on the right). 

	To realise the Beilinson--Bernstein correspondence, we need first to compute the map $$\Phi_\lambda : U(\g) \longrightarrow \G(X, \Dl)$$ (see section 1) which specifies the $\g$-module structure on $\G(X, \Dl)$. Using the definition sheaves as equalisers, we may write
	\begin{equation}\label{localglobal}
		\G(X, \Dl) = \{(P,Q) \in \G(U_0, \Dl) \times \G(U_\infty, \Dl) : P = \Psi_\lambda(Q)\}
	\end{equation} so it will suffice to describe the image of $\Phi_\lambda$ on the charts $U_0$ and $U_\infty$. Using equations (\ref{diffexp}) and (\ref{glue}), we may compute the action of $\Phi_\lambda$ on the standard basis of $\g$, finding that 
	\begin{equation}\label{global}
		\Phi_\lambda(E) = (E_0, E_\infty), \qquad \Phi_\lambda(F) = (F_0, F_\infty), \qquad \Phi_\lambda(H) = (H_0, H_\infty)
	\end{equation} where 
	\begin{equation}\label{infinitesimal}
		E_0 = z^2\p^z - z(\lambda -1), \qquad F_0 = -\p^z, \qquad H_0 = 2z\p^z - (\lambda-1)
	\end{equation}
	and
	\begin{equation}\label{infinity}
		E_\infty = -\p^w, \qquad F_\infty = w^2\p^w - w(\lambda -1), \qquad H_\infty = -2w\p^w + (\lambda-1)
	\end{equation}
	and we refer the reader to section 4 of \cite{romanov} for the details of this calculation.
	\section{At most $2$ regular singularities}
	In this section we provide a characterisation of the $\slt$-representations (with dominant regular central character) that correspond to holonomic $\D$-modules on $X = \mathbb{CP}^1$ with at most 2 regular singularities via Beilinson--Bernstein localisation. We will proceed topologically: we will first classify all the monodromy representations with the prescribed amount of singularities, then translate to holonomic regular $\D$-modules using the Riemann--Hilbert correspondence, and finally to $\slt$-representations via Beilinson--Bernstein localisation. 
	\subsection*{The case of $\Oc$-coherent $\D$-modules with no singularities}
	The simplest case is that of $\D$-modules with no singularities at all. Via the Riemann--Hilbert correspondence (see the argument in section 2 for details), these correspond to finite-dimensional $\pi_1(X)$-representations; since $X = \mathbb{CP}^1$ is simply connected, such a representation consists only of a choice of finite-dimensional complex vector space $L \cong \mathbb{C}^n$ with trivial action of $\pi(X) = 0$. As per (\ref{inverse RH}), this retrieves the modules $$\M_n = \Ox \otimes_{\mathbb{C}_X} \mathbb{C}_X^n \cong \Ox^{\oplus n}$$ whereon $\D_X$ acts as $\nabla(f_1, \ldots, f_n) = (df_1, \ldots, df_n)$. In the twisted case, the analogous argument yields the $\Oll$-coherent $\Dl$-modules $$\M_{\lambda,n} = \Oll^{\oplus n}$$ whereon $\Dl$ acts via the twist $\Psi_\lambda$. These are all the possible $\Oll$-coherent $\Dl$-modules on $X$.
	
	We would now like to describe the $\g$-representations that these $\Dx$-modules correspond to under Beilinson--Bernstein localisation. Since the global sections functor commutes with finite direct sums, we will describe only the case $n=1$; the others will follow as the direct sum of these representations:
	\begin{equation}
	    \G(X,\M_{\lambda, n}) \cong \bigoplus_{i=1}^n\G(X,\Oll). \label{ndirsum}
	\end{equation} Using the definition of sheaves as equalisers, we find that
	\begin{align}
		\G(X, \Oll) &= \{(f,g) \in \G(U_0, \Oll) \times \G(U_\infty, \Oll) : f = \psi_\lambda(g)\}\nonumber \\
		&= \{(f,g) \in \mathbb{C}[z] \times \mathbb{C}[w] : f = \psi_\lambda(g)\} \label{equaliser}
	\end{align} where $f = \psi_\lambda(g)$ unfolds as $f(z) = z^{\lambda-1}g(w) = z^{\lambda-1}g(z^{-1})$ using equation (\ref{small psi}). For this condition to be satisfied, $g$ must be of degree at most $\lambda - 1$, and to any such $g$ we may associate a unique $f \in \mathbb{C}[z]$ satisfying the condition. Thus we can identify $$\G(X, \Oll) \cong \mathbb{C}[w]_{\leq \lambda-1} = \bigoplus_{k=0}^{\lambda -1} \mathbb{C} \cdot w^k$$ and $\g$ acts on $\G(X, \Oll)$ using the description in equation (\ref{infinity})\footnote{A priori, $\g$ acts on the global sections via equation (\ref{global}), with (\ref{infinitesimal}) corresponding to the chart $U_0$ and (\ref{infinity}) to the chart $U_\infty$. However, in this case, global sections are uniquely determined by their restriction to $U_\infty$, so it suffices to describe the $\g$-module structure there. The general distinction will become clear in our subsequent examples.}. Explicitly, we obtain
	\begin{equation}
		E_\infty = -kw^{k-1}, \qquad F_\infty = (k-\lambda+1)w^{k+1}, \qquad H_\infty = (\lambda-1 -2k)w^k \qquad \label{findimeq}
	\end{equation}
	and we visualise this action in Figure \ref{fig:findim}. We identify this representation as the unique irreducible $\slt$-module of dimension $\lambda$ which is given by the highest weight. We conclude:
	\begin{theorem}
		Beilinson--Bernstein localisation restricts to an equivalence of:
		\begin{enumerate}[(a)]
			\item the category of $\Oll$-coherent $\Dl$-modules on $X = \mathbb{CP}^1$ with no singularities; and
			\item the subcategory of $\slt$-representations with central character $\chi_\lambda$ which is generated by the irreducible finite-dimensional representation of highest weight $\lambda -1$ via direct sums.
		\end{enumerate} 
	\end{theorem}
	\begin{figure}
		\centering
		\[
		\begin{tikzcd}
			1^{\textcolor{white}{0}} \arrow[loop above, "\lambda-1", color = Cyan] \arrow[r, "1-\lambda", color = Green, bend left = 15]&
			w^{\textcolor{white}{1}} \arrow[loop above, "\lambda-3", color = Cyan]  \arrow[l, "-1", color = Red, bend left = 15] \arrow[r, "2-\lambda", color = Green, bend left = 15]&
			w^2 \arrow[loop above, "\lambda-5", color = Cyan] \arrow[l, "-2", color = Red, bend left = 15]\arrow[r, "3-\lambda", color = Green, bend left = 15] &
			{\ldots \, \,} \arrow[l, "-3", color = Red, bend left = 15] \arrow[r, "-2", color = Green, bend left = 15]&
			w^{\lambda -1}  \arrow[loop above, "1-\lambda", color = Cyan]\arrow[l, "1-\lambda", color = Red, bend left = 15]
		\end{tikzcd}
		\]
		\caption{The $\g$-module structure on $\G(X,\Oll)$. Coloured arrows correspond to the action of the $\slt$-basis \textcolor{Red}{$E$},\textcolor{Green}{$F$} and \textcolor{Cyan}{$H$} on the basis $\{w^k\}_{k=0}^{\lambda - 1}$ of $\G(X,\Oll) \cong \G(U_\infty, \Oll)$, with multiplicity written as the arrows' labels.}
		\label{fig:findim}
	\end{figure}
	
	\subsection*{The case of one regular singularity}
	As a next step, we'd like to classify the $\slt$-representations which correspond to holonomic $\Dl$-modules with exactly $1$ singularity. Let $\M$ be a holonomic $\D$-module with exactly 1 regular singularity, and assume w.l.o.g. that this singularity lies at the point $0 = [1:0]$. It will then be natural for us to work on the stratification
	\begin{equation}
	    X = U_\infty \amalg \{0\}. \label{stratfcone}
	\end{equation}
	Via the Riemann--Hilbert correspondence (see section 2), $\M$ corresponds to a choice of $\pi_1(U_\infty)$-representation $\M_U$ representing the restriction of $\M$ and a choice of finite-dimensional complex vector space $M_0$ representing the stalk at $0$. As $U_\infty = X \backslash \{0\} \cong \mathbb{C}$ is simply connected, an argument analogous to the one in the preceding subsection reveals that $\M_U \cong \Oll^{\oplus n}$ for some $n \in \mathbb{N}$.
	
	By Theorem \ref{comp series}, the category of holonomic $\Dl$-modules $\M$ with but one regular singularity at $0$ is generated via composition series by the $\Dl$-modules $\E(U_\infty, j_\infty^\dagger \M)$ and $\E(\{0\}, i_0^\dagger \M)$ (recall that notation for all embeddings has been fixed in Section 3). The former is a proper submodule of $(j_\infty)_+(\Oll)$, while the latter is isomorphic to $(i_0)_+(\mathbb{C})$ by Lemma \ref{adjoint}(d). To describe $\slt$-representations with exactly one regular singularity, it will hence suffice to describe the $\g$-representation structure on $(j_\infty)_+(\Oll)$  and $(i_0)_+(\mathbb{C})$ obtained from the Beilinson--Bernstein correspondence.
	 
	 We begin with $\M_0 = (i_0)_+(\mathbb{C})$. As this module is entirely supported in the chart $U_0$, we must have $$\G(X, \M_0) \cong \G(U_0, \M_0)$$ and it will suffice to work in the chart $U_0$ only. Unfolding the definition of the direct image functor (cf. the introduction), we find:
	 \begin{align}
	 	 \G(U_0, (i_{00})_+(\mathbb{C})) &= \G(U_0, \D_{U_0 \leftarrow V}) \otimes_{\mathbb{C}} \mathbb{C} \nonumber \\
	 	 &\cong \mathbb{C} \otimes_{\G(U_0, \Oc_{U_0})} \G(U_0,\D_{U_0}) \nonumber \\
	 	 &\cong \frac{\mathbb{C}[z]}{z \cdot \mathbb{C}[z]} \otimes_{\mathbb{C}[z]} \mathbb{C}[z, \p^z] \nonumber \\
	 	 &\cong \frac{\mathbb{C}[z, \p^z]}{z \cdot \mathbb{C}[z, \p^z]} \nonumber \\
	 	 &\cong \bigoplus_{k=0}^\infty \p_z^k \delta_0 \label{deltazero}
	 \end{align} where $\delta_0$ denotes the Dirac indicator function at $z = 0$. Here $\D_{U_0}$ acts by multiplication on the right. Choosing the basis 
 	\begin{equation}
 		\left\{\delta_0^{(k)} = \frac{1}{k!} \p_z^k \delta_0\right\}_{k=0}^\infty \label{basis delta}
 	\end{equation} we apply equation (\ref{infinitesimal}) to obtain the $\g$-action on this vector space. We refer the reader to Section 6 of \cite{romanov} for details of the calculation. We obtain:
 	$$E \cdot \delta_0^{(k)} = (\lambda + k)  \delta_0^{(k-1)}, \qquad
 	F \cdot \delta_0^{(k)} = -(k+1)  \delta_0^{(k+1)}, \qquad
 	H \cdot \delta_0^{(k)} = -(\lambda  1 + 2k)  \delta_0^{(k)} $$
 	for $k > 0$ and
 	$$E \cdot \delta_0 = 0, \qquad
 	F \cdot \delta_0= 0 ,\qquad
 	H \cdot \delta_0= -(\lambda + 1)  \delta_0^{(k)}$$
 	for $k=0$. We illustrate the result in Figure \ref{fig:Verma} and identify this as the irreducible Verma module of highest weight $-\lambda-1$.\\
 	
 	\begin{figure}
 		\centering
 		\[
 		\begin{tikzcd}
 			\delta_0^{\textcolor{white}{(0)}} \arrow[loop above, "-\lambda-1", color = Cyan] \arrow[r, "-1", color = Green, bend left = 15]&
 			\delta_0^{(1)} \arrow[loop above, "-\lambda-3", color = Cyan]  \arrow[l, "\lambda+1", color = Red, bend left = 15] \arrow[r, "-2", color = Green, bend left = 15]&
 			\delta_0^{(2)} \arrow[loop above, "-\lambda-5", color = Cyan] \arrow[l, "\lambda+2", color = Red, bend left = 15]\arrow[r, "-3", color = Green, bend left = 15] &
 			\delta_0^{(3)} \arrow[l, "\lambda + 3", color = Red, bend left = 15] \arrow[r, "-4", color = Green, bend left = 15]  \arrow[loop above, "-\lambda-7", color = Cyan]&
 			{\ldots \, \,} \arrow[l, "\lambda + 4", color = Red, bend left = 15]
 		\end{tikzcd}
 		\]
 		\caption{The $\g$-module structure on $\G(X,\M_0)$. Coloured arrows correspond to the action of the $\slt$-basis \textcolor{Red}{$E$},\textcolor{Green}{$F$} and \textcolor{Cyan}{$H$} on the basis (\ref{basis delta}) of $\G(X,\M_0) \cong \G(U_0,\M_0)$, with multiplicity written as the arrows' labels.}
 		\label{fig:Verma}
 	\end{figure}
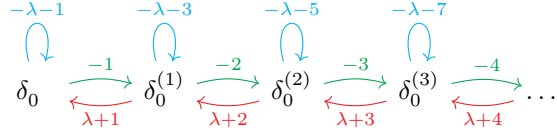
 
  	Consider now the $\Dx$-module $\M_U = (j_\infty)(\Oll)$ attached to the open stratum $U_\infty$. In the chart $U_\infty$, we simply have $$\G(U_\infty, (j_\infty)_+(\Oll)) = \G(U_\infty, \Oll) \cong \mathbb{C}[w]$$ with the same action as in the example in the preceding section, i.e. given by equation (\ref{findimeq}). The difference here is of course that this representation has infinite dimension. Contrary to the example of $(i_{00})_+(\mathbb{C})$, this module is supported in both charts, so a priori we should also describe its $\g$-action on the other chart, in this case $U_0$. However, note that
  	\begin{align*}
  		\G(U_0, \M_U) &= (j_{00})_+(j_{\infty\infty})_+(\G(U_\infty, \Oll)) \\
  		&= (j_{00})_+(\G(\ins, \Oll) \otimes_{\G(U_\infty, \Oll)}\G(U_\infty, \Oll)) \\
  		&= (j_{00})_+(\mathbb{C}[w,w^{-1}] \otimes_{\mathbb{C}[w]} \mathbb{C}[w]) \\
  		&= (j_{00})_+(\mathbb{C}[w,w^{-1}]) \\
  		&= \mathbb{C}[w,w^{-1}]
  	\end{align*}
  	as a complex vector space. Using the same argument as in (\ref{equaliser}), we may deduce:
  	\begin{align*}
  		\G(X, \M_U) &= \{(f,g) \in \G(U_0, \M_U) \times \G(U_\infty, \M_U) : f = \psi_\lambda(g)\} \\
  		&\cong \{(f,g) \in \mathbb{C}[z,z^{-1}] \times \mathbb{C}[w] : f(z^{-1}) = g(z)\} \\
  		&\cong \mathbb{C}[w]
  	\end{align*}
  where the last isomorphism is projection onto the second factor. Hence again the global sections are entirely determined by the sections on one chart---this time $U_\infty$---and there is no need for us to compute the action of $\g$ on the sections on $U_0$. Instead we note that as $\g$-modules, we have $$\G(X, \M_U) \cong \G(U_\infty, \M_U) \cong \mathbb{C}[w]$$ with $\g$-action given by equation (\ref{findimeq}). We illustrate this $\slt$-representation in Figure \ref{fig:dualVerma} and identify it as the reducible dual Verma module of highest weight $\lambda - 1$. 
  
  Note in particular that this module has a unique submodule spanned by the set $\{1,w,\ldots, w^{\lambda-1}\}$ which is isomorphic to the irreducible finite-dimensional representation of highest weight $\lambda - 1$, constructed in the preceding subsection (cf. Figure \ref{fig:findim}). This reflects the fact that unlike in the case of the closed embedding $i_0$, the (improper) open embedding $j_\infty$ induces different functors $\int_{i_0!}$ and $\int_{i_0}$, so the minimal extension is a proper submodule of the direct image.
  
  As the dual Verma has a unique irreducible submodule, this completes the characterisation. We have arrived at:
  
  \begin{absolutelynopagebreak}
    \begin{theorem} \label{onesing}
  	 Let $\lambda \in \mathbb{Z}_{>0}$ and $\M$ be a holonomic $\Dl$-module on $X = \mathbb{CP}^1$ with exactly one regular singularity. Consider the $\slt$-representation $V = \G(X,\M)$ corresponding to $\M$. There exists a composition series of $\slt$-representations $$0 \leq V_n \leq V_{n-1} \leq \ldots \leq V_1 = V$$ such that each quotient $\frac{V_i}{V_{i+1}}$ is one of the following:
  	\begin{enumerate}[(a)]
  		\setlength\itemsep{0.1em}
  		\item the irreducible Verma module of highest weight $-\lambda - 1$;
  		\item the irreducible finite-dimensional representation of highest weight $\lambda - 1$. 
  	\end{enumerate}
  \end{theorem}
   \end{absolutelynopagebreak}
 \begin{proof}
    We may assume w.l.o.g. that $\M$ has no singularities away from the point $0 = [1:0]$. By Theorem \ref{comp series}, any holonomic $\Dx$-module with one regular singularity at $\{0\}$ admits a composition series whose quotients are minimal extensions of $\Oc$-coherent $\D$-modules on the strata of (\ref{stratfcone}). By Lemma \ref{equivtwist}, the same holds for $\Dl$-modules in the case where $\lambda$ is an integer. Hence there exists a composition series of $\Dl$-modules
    \begin{equation*}
        0 \trianglelefteq \M_n \trianglelefteq \M_{n-1} \trianglelefteq \ldots \trianglelefteq \M_1 = \M \label{modcompser}
    \end{equation*} such that each quotient $\frac{\M_i}{\M_{i+1}}$ is a minimal extension of a $\D$-module on $U_\infty$ or $\{0\}$. By the Beilinson--Bernstein correspondence (Theorem \ref{BBL}), the functor $\G(X, -)$ is an equivalence of categories, and in particular it is an exact functor. This means it preserves inclusions and quotients, so applying $\G(X, -)$ to (\ref{modcompser}) yields the composition series of $\slt$-representations $$0 \leq \G(X, \M_n) \leq \G(X, \M_{n-1}) \leq \ldots \leq \G(X, \M_{1}) = \G(X, \M)$$ whose quotients are the representations corresponding to the (simple) minimal extensions of $\Oc$-coherent $\D$-modules on $U_\infty$ or $\{0\}$. In the latter case, this retrieves the representation in Figure \ref{fig:Verma}, which is the irreducible Verma module of highest weight $-\lambda - 1$, i.e. option (a). In the former case, it must be an irreducible submodule or irreducible quotient of the representation in Figure \ref{fig:dualVerma}, which is the reducible dual Verma module of highest weight $\lambda - 1$. This representation has a unique irreducible submodule, option (b), and we find that its unique irreducible quotient is isomorphic to the irreducible Verma module of highest weight $-\lambda - 1$, i.e. option (a) again.
 \end{proof}
  \begin{figure}
  	\centering
  	\[
  	\begin{tikzcd}
  		1^{\textcolor{white}{0}} \arrow[loop above, "\lambda-1", color = Cyan] \arrow[r, "1-\lambda", color = Green, bend left = 15]&
  		w^{\textcolor{white}{1}} \arrow[loop above, "\lambda-3", color = Cyan]  \arrow[l, "-1", color = Red, bend left = 15] \arrow[r, "2-\lambda", color = Green, bend left = 15]&
  		w^2 \arrow[loop above, "\lambda-5", color = Cyan] \arrow[l, "-2", color = Red, bend left = 15]\arrow[r, "3-\lambda", color = Green, bend left = 15] &
  		w^3 \arrow[loop above, "\lambda-7", color = Cyan] \arrow[l, "-3", color = Red, bend left = 15]\arrow[r, "4-\lambda", color = Green, bend left = 15] &
  		{\ldots \, \,} \arrow[l, "-4", color = Red, bend left = 15]
  	\end{tikzcd}
  	\]
  	\caption{The $\g$-module structure on $\G(X,\M_U)$. Coloured arrows correspond to the action of the $\slt$-basis \textcolor{Red}{$E$},\textcolor{Green}{$F$} and \textcolor{Cyan}{$H$} on the basis $\{w^k\}_{k=0}^{\lambda - 1}$ of $\G(X,\M_U) \cong \G(U_\infty,\M_U)$, with multiplicity written as the arrows' labels.}
  	\label{fig:dualVerma}
  \end{figure}
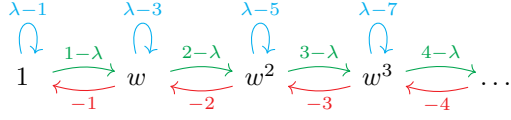
 
 \subsection*{The case of two regular singularities}
 Finally, we tend to the most general case of two regular singularities, w.l.o.g. at the points $0 = [1:0]$ and $\infty = [0:1]$. The natural choice of stratification is then \begin{equation}
     X = \{0\} \amalg (\ins) \amalg \{\infty\} \label{stratfc}
 \end{equation} and we adopt the same strategy for studying this case as in the previous example, i.e. computing the direct image of $\D$-modules supported at each stratum.
  Observe that this stratification corresponds to the orbits of the action by the Lie subgroup
  $$K = \left\{ \begin{pmatrix}
  	a & 0 \\
  	0 & a^{-1}
  \end{pmatrix} : \mathbb{C}^\times \right\} \cong \mathbb{C} \otimes_\mathbb{Z} \operatorname{SL}_2(\mathbb{R}) \subseteq G$$
which acts by restriction of the $G$-action; consequently, we can expect to characterise the corresponding representations as certain admissible representations of $\operatorname{SL}_2(\mathbb{R})$ (cf. \cite{milicic} and \cite{lusztig}).

The case of the closed stratum $\{0\}$ is the same as the argument in the previous subsection (recall that the global sections were entirely determined on $U_0$, so nothing has changed from the perspective of this stratum). The only $\D$-module associated to this stratum is $(i_0)_+(\{0\})$ and the $\slt$-representation structure on the global sections thereof is shown in Figure \ref{fig:Verma}. This irreducible Verma module turns out to be the Harish-Chandra module of the irreducible holomorphic discrete series representation of $\operatorname{SL}_2(\mathbb{R})$ of highest weight $-\lambda -1$.

 Symmetrically, the global sections of the $\D$-module attached to the other closed stratum $\{\infty\}$, that is, $(i_\infty)_+(\{\infty\})$, admit as basis the derivatives of the Dirac distribution $\delta_\infty$ at the point $\infty$, and its structure is that of the representation at $\{0\}$, merely interchanged by $\Psi_\lambda$. We spare the reader the details of this (almost identical) calculation and exhibit its result in Figure \ref{fig:antiholomorphic}. We can identify this representation as the graded dual of the irreducible Verma module of highest weight $-\lambda -1$. Its weights are bounded from below, so we say that it has \textit{lowest weight} $\lambda + 1$. This representations turns out to correspond to the irreducible anti-holomorphic discrete series representation of $\operatorname{SL}_2(\mathbb{R})$ of lowest weight $\lambda + 1$. 

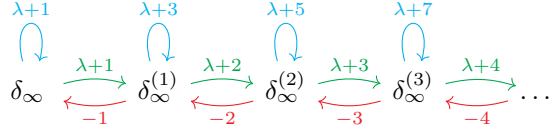
\begin{figure}
	\centering
	\[
	\begin{tikzcd}
		\delta_\infty^{\textcolor{white}{(0)}} \arrow[loop above, "\lambda+1", color = Cyan] \arrow[r, "\lambda+1", color = Green, bend left = 15]&
		\delta_\infty^{(1)} \arrow[loop above, "\lambda+3", color = Cyan]  \arrow[l, "-1", color = Red, bend left = 15] \arrow[r, "\lambda+2", color = Green, bend left = 15]&
		\delta_\infty^{(2)} \arrow[loop above, "\lambda+5", color = Cyan] \arrow[l, "-2", color = Red, bend left = 15]\arrow[r, "\lambda+3", color = Green, bend left = 15] &
		\delta_\infty^{(3)} \arrow[l, "-3", color = Red, bend left = 15] \arrow[r, "\lambda+4", color = Green, bend left = 15]  \arrow[loop above, "\lambda+7", color = Cyan]&
		{\ldots \, \,} \arrow[l, "-4", color = Red, bend left = 15]
	\end{tikzcd}
	\]
	\caption{The $\g$-module structure on $\G(X,(i_\infty)_+(\mathbb{C}))$. Coloured arrows correspond to the action of the $\slt$-basis \textcolor{Red}{$E$},\textcolor{Green}{$F$} and \textcolor{Cyan}{$H$} on a basis of $\G(X,(i_\infty)_+(\mathbb{C})) \cong \G(U_0,(i_\infty)_+(\mathbb{C}))$, with multiplicity written as the arrows' labels.}
	\label{fig:antiholomorphic}
\end{figure}

Finally, we consider the open stratum $\ins \cong \mathbb{C}^\times$.  Via the Riemann--Hilbert correspondence, $\Oll$-coherent $\Dl$-modules on $\ins$ are indexed by finite-dimensional complex representations $$\rho : \pi_1(\ins) \cong \mathbb{Z} \longrightarrow  \GL(L).$$ Now $\pi_1(\ins) \cong \mathbb{Z} = \langle 1\rangle$, so the data of such a representation is equivalent to a choice of vector space $L \cong \mathbb{C}^n$ and an invertible $n\times n$ complex matrix $T = \rho(1)$. A universal covering space of $\ins = \mathbb{C}^\times$ is given by the complex exponential map $\operatorname{exp} : \mathbb{C} \to \mathbb{C}^\times$. Now $ \pi_1(\ins) \cong \mathbb{Z}$ acts on $\mathbb{C}$ via deck transformations, i.e. as $$n \cdot z = z + 2i \pi  n$$ and on $L = \mathbb{C}^n$ as $\rho(n) = T^n$. Let $\M_T$ be the corresponding $\Oc$-coherent $\D$-module on $\ins$, i.e. $$\M_T = \Oll \otimes_{\mathbb{C}} L \cong \Oll^{\oplus n}$$ by \ref{inverse RH}. We aim to find the $\D$-module structure on $\M_T$ explicitly by constructing an inverse to the de Rham functor in this context. Recall that $\M_T$ is described fully as an $\Oc$-module with connection $$\nabla_T : \Oll^{\oplus n} = \M_T \longrightarrow \Omega^1_{\ins} \otimes_{\Oll} \M_T \cong (\Omega^1_{\ins})^{\oplus n}.$$ As $\ins$ is contained fully in both charts, we may work locally, say on the chart $U_0$, where we can choose a trivialisation $$\nabla_T \cong \operatorname{d} - \frac{A(z) \operatorname{dz}}{z}$$ for some matrix $A(z)$ whose entries are regular functions. An analogous result holds on $U_\infty$. Via the Riemann--Hilbert correspondence, this corresponds to the local system $$\ker(\nabla_T) = \{(f_1, \ldots, f_n) : (df_1, \ldots, df_n) = A\cdot \tfrac{1}{z}(f_1, \ldots, f_n)\}$$ which, morally, is the sheaf of solutions to the differential equation $\operatorname{d}\vec{z} = A(z)\vec{z}$. Embedding $\Oll$ in the endomorphism sheaf as in Section 1 (via the map $f \mapsto f \cdot - $), we can equivalently view each $f_i$ as an $n \times n$ matrix $F_i$ with entries in $\Oll$, i.e. $$\ker(\nabla_T) = \{(F_1, \ldots, F_n) : (dF_1, \ldots, dF_n) = A \cdot(zI_n)^{-1}(F_1, \ldots, F_n)\}.$$ We note that the matrix-valued function $z^A = \exp(A\log(z))$ is a section of this sheaf, as $\frac{d}{dz}z^A = Az^{A-I}$ holds (here we are using the complex matrix exponential and logarithm).

Now, $T$ acts on the stalks of this sheaf as the monodromy sections obtained by going around the origin on the counter-clockwise path $t \mapsto e^{2 \pi i t}$. Any two sections will have conjugate monodromy, and conjugate monodromies must correspond to the same connection, since monodromy is precisely a change of basis for $\ker(\nabla_T)$. Thus we may take $T = \exp(2 \pi i A)$, and conversely, $A = \log(T)$ where we are choosing any branch of the multi-valued complex matrix logarithm. We have obtained the $\Oc$-coherent $\D$-module $ \M_T = \Oll^{\oplus n}$ on $\ins$, with connection given in terms of the chart $U_0$ as 
\begin{equation}
    \nabla = d - \frac{\log(T)\operatorname{dz}}{2\pi i \cdot } \label{log}
\end{equation} where $\log$ denotes the complex matrix logarithm.
Now, we'd like to compute the $\slt$-representation which corresponds to the direct image $(j_0)_+(\M_T)$ via Beilinson--Bernstein localisation. We may still work entirely within the chart $U_0$: $j_0$ is an open immersion, so the $\D$-module direct image functors agree with the sheaf-theoretic push-forward. Thus we have:
\begin{align*}
    \G(U_0, (j_{0\infty})_+(\Oll^{\oplus n})) &\cong \G(U_0, (j_{0})_+(\Oll^{\oplus n})) \\
    &\cong \G(U_0, \Oll^{\oplus n}) \\
    &\cong \mathbb{C}[z, z^{-1}]^{\oplus n}
\end{align*}
as sheaves of vector spaces. Similarly, on the chart $U_\infty$, we have $$\G(U_\infty,(j_{0\infty})_+(\Oll^{\oplus n})) \cong \mathbb{C}[w, w^{-1}]^{\oplus n}$$ and the action of $\Dx$ on each vector space is inherited from the action of $\Dx$ on the intersection. These local sections are naturally isomorphic, so indeed $$\G(X, (j_{0\infty})_+(\Oll^{\oplus n})) \cong \mathbb{C}[z,z^{-1}]^{\oplus n}$$ and the analogous statement with $w$ instead of $z$ holds as well.

To compute the $\slt$-representation structure, we may thus use the formulae (\ref{infinitesimal}) with (\ref{log}) substituted for $\p^z$. In the case where the rank is $n = 1$, our monodromy generator $T$ is just a non-zero complex scalar, and we find:
	$$E \cdot z^k = (k+1- \lambda - \tau)z^{k+1}, \quad
 	F \cdot z^k = (\tau - k)z^{k-1}, \quad
 	H \cdot z^k = (2k+1 - \lambda - 2\tau)z^k $$
 where $$\tau =\frac{\log(T)}{2\pi i}$$ and $\log(-)$ denotes any branch of the complex logarithm. We illustrate this representation $V = \G(X, (j_{0\infty})_+(\Oc_{\ins}))$ in Figure \ref{fig:Zl}.
 
  Note first that this $V$ is irreducible unless $\tau$ is an integer, which happens precisely when $T$ is the trivial monodromy. If $\tau$ is indeed an integer then $V$ has an irreducible submodule spanned by $\{z^\tau, \ldots, z^{\lambda  + \tau - 1}\}$ which we identify as the irreducible finite-dimensional representation of highest weight $\lambda - 1$. The quotient of $V$ by this submodule is the direct sum of the representations in Figures \ref{fig:Verma} and \ref{fig:antiholomorphic}, both of which are isomorphic to the irreducible Verma module of highest weight $-\lambda - 1$.
  
  In particular, if $T = \pm I_n$ then $\log(T) = 0$ or $\log(T) = \pi i$, so this monodromy correspond to the connections $$\nabla = d \qquad \text{and} \qquad \nabla = d - \frac{dz}{2z}$$ whose resulting representations are given by Figure \ref{fig:Zl} with $\tau = 0$ and $\tau = \tfrac{1}{2}$, respectively. These are precisely the Harish-Chandra modules associated to the reducible (resp. irreducible) principal series representations of $\operatorname{SL}_2(\mathbb{R})$, as outlined in \cite{romanov}.
  
  For this reason, we shall refer in to case of the representation in Figure \ref{fig:Zl} as the \textit{quasi-admissible} representation of $\slt$ associated to $\tau \in \mathbb{C}$. In the case of $\tau \in \mathbb{Z}$, we call it the (reducible) trivial admissible representation. We are now ready to state and prove our final result:
 
 \begin{theorem}
       Let $\lambda \in \mathbb{Z}_{>0}$ and $\M$ be a holonomic $\Dl$-module on $X = \mathbb{CP}^1$ with two singularities. Consider the $\slt$-representation $V = \G(X,\M)$ corresponding to $\M$. Then there is a decomposition of $\slt$-representations $$V \cong V^{{\neq 1}} \oplus V^{{=1}}$$ where $V^{{\neq 1}} = \G(X, (j_{0\infty})_+(\mathcal{N}))$ for some $\Oc_{\ins}$-coherent $\D_{\ins}$-module $\mathcal{N}$ and $V^{{=1}}$ admits a composition series of $\slt$-representations $$0 \leq V_n \leq V_{n-1} \leq \ldots \leq V_1 = V$$ whose quotients are one of the following:
       \begin{enumerate}[(a)]
  		\setlength\itemsep{0.1em}
  		\item the irreducible Verma module of highest weight $-\lambda - 1$; or
  		\item the graded dual thereof, with lowest weight $\lambda +1$; or
  		\item the irreducible finite-dimensional representation of highest weight $\lambda - 1$.
  	\end{enumerate}
 \end{theorem}
 \begin{proof}
    Let $\M$ be a holonomic $\Dx$-module with exactly two singularities and assume w.l.o.g. that these lie at the points $0 = [1:0]$ and $\infty = [0:1]$. Via the Riemann--Hilbert correspondence, the restriction of $\M$ to $\ins$ corresponds to a local system $\Ll$ on $\ins$, i.e. a monodromy representation $$\rho : \pi_1(\ins) \cong \mathbb{Z} \longrightarrow \GL(\mathbb{C}^m)$$ generated by some $\rho(1) = T \in \GL(\mathbb{C}^m)$. Fix a point $x \in \ins$ and a stalk $\Ll_x$; now $T$ acts on $\Ll_x$ as monodromy (see Section 2). As locally constant sheaves are fixed by their sections, we may extend this action to all sections. Using Jordan decomposition of the matrix $T$, we may write $$\mathcal{L} = \mathcal{L}^{{\neq 1}} \oplus \mathcal{L}^{{=1}}$$ where $\mathcal{L}^{{=1}}$ is the generalised eigenspace of $T$ corresponding to the eigenvector $1$ and $\mathcal{L}^{{\neq 1}}$ is the direct sum of the other generalised eigenspaces. Thus, any (simple) quotient in the composition series of $\M$ (given in Theorem \ref{comp series} in the untwisted case, and via Lemma \ref{equivtwist} also in the case of a twist by integral $\lambda$) whose support is not contained in $\{0\}$ or $\{\infty\}$ must correspond to a local system contained in exactly one of $\Ll^{{\neq 1}}$ or $\Ll^{{=1}}$. Observe that $\operatorname{Ext}^1(\mathcal{N},\mathcal{N}') = 0$ if $\N$ is a $\D$-module with nontrivial monodromy and $\mathcal{N}'$ is a $\D$-module supported at $0$ or $\infty$. Passing again through the Riemann--Hilbert correspondence (the de Rham functor must be exact as it is an equivalence), we obtain a decomposition
    \begin{equation}
        \M = \M^{{\neq 1}} \oplus \M^{{=1}} \label{dirsum}
    \end{equation}
     where $\M^{{\neq 1}}$ has a composition series whose quotients are the minimal extensions of $\Oll$-coherent $\Dl$-modules on $\ins$ with non-trivial monodromy, and $\M^{{=1}}$ has a composition series whose quotients are the $\Oll$-coherent $\Dl$-modules on $\ins$ with trivial monodromy as well as the modules supported on the closed strata.
    Taking global sections yields $$V = \G(X, \M) \cong \G(X, \M^{{\neq 1}}) \oplus \G(X, \M^{{=1}})$$ 
    as $\G(X, -)$ is an exact functor by Beilinson--Bernstein localisation. Taking $V^{{=1}} = \G(X, \M^{{\neq 1}})$ and $V^{{\neq 1}} = \G(X, \M^{{\neq 1}})$ gives the postulated decomposition. Moreover, taking global sections of the composition series of $\M^{{=1}}$, we obtain a composition series of $\slt$-representations $$0 \leq \G(X, \M_n) \leq \G(X, \M_{n-1}) \leq \ldots \leq \G(X, \M_{1}) = \G(X, \M^{{=1}}) = V^{{=1}}$$ whose quotients are the representations corresponding to the (simple) minimal extensions of $\Oc$-coherent $\D$-modules on the strata $\ins$ or ${0}$ or $\{\infty\}$, or a union thereof, and those on $\ins$ must have trivial monodromy. As argued above, the minimal extension of an $\Oll$-coherent $\Dl$-module on $\ins$ whereon $T$ acts as identically must be a direct sum of modules of the form (a) or (c). The representations arising from $\D$-modules on the strata $\{0\}$, $U_0$ and $U_\infty$ must be of the forms (a) or (c) using the same reasoning as in Theorem \ref{onesing}, and for the stratum $\{\infty\}$ it must be of the form (b), by the argument above.
 \end{proof}

\begin{figure}
	\centering
	\[
	\begin{tikzcd}
		{\ldots \, \,} \arrow[r, "-2-\lambda-\tau", color = Red, bend left = 15]&
		z^{-2} \arrow[loop above, "-(\lambda+4)", color = Cyan] 
			\arrow[l, "\tau + 2", color = Green , bend left = 15]
			\arrow[r, "-1-\lambda-\tau", color = Red, bend left = 15]&
		z^{-1} \arrow[loop above, "-(\lambda+2)", color = Cyan]
			\arrow[l, "\tau + 1", color = Green , bend left = 15]
			\arrow[r, "-\lambda-\tau", color = Red, bend left = 15] &
		z^0 \arrow[loop above, "-\lambda", color = Cyan]
			\arrow[l, "\tau", color = Green , bend left = 15]
			\arrow[r, "1-\lambda-\tau", color = Red, bend left = 15]&
		z^{1} \arrow[loop above, "2-\lambda", color = Cyan] 
			\arrow[l, "\tau - 1", color = Green , bend left = 15]
			\arrow[r, "2-\lambda-\tau", color = Red, bend left = 15]&
		z^{2} \arrow[loop above, "4-\lambda", color = Cyan]
			\arrow[l, "\tau - 2", color = Green , bend left = 15]
			\arrow[r, "3-\lambda-\tau", color = Red, bend left = 15] &
		{\ldots \, \,} \arrow[l, "\tau - 3", color = Green , bend left = 15]
	\end{tikzcd}
	\]
	\caption{The $\slt$-module structure on the vector space of global sections $\G(X, (j_{0\infty})_+(\Oc_{\ins})) \cong \mathbb{C}[z,z^{-1}]$. Coloured arrows correspond to the action of the $\slt$-basis \textcolor{Red}{$E$},\textcolor{Green}{$F$} and \textcolor{Cyan}{$H$} with multiplicity written as the arrows' labels.}
	\label{fig:Zl}
\end{figure}
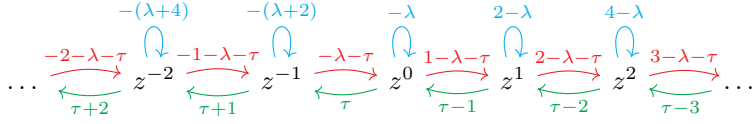

\section{Further regular singularities}

	In this last section, we look at some examples of the many $\slt$-representations which arise from holonomic $\D$-modules on $X = \mathbb{CP}^1$ with more than two regular singularities. In general, this category is very large. Indeed---as we saw in Section 2---given a finite set $S = \{s_1, \ldots, s_m\}$ of points on $X$, any finite-dimensional complex $\pi_1(X \backslash S)$-representation induces a local system on $X \backslash S$ of finite rank. Via the Riemann--Hilbert correspondence, this in turn gives rise to $\Dx$-modules with (regular) singularities in $S$. In other words, we may construct $\D$-modules on $X$ with any prescribed monodromy around the points in $S$. 
	Let us discuss this in some more detail.
	
	As $X \backslash S$ is homeomorphic to the complex plane punctured at $|S|-1$ points, it is homotopy equivalent to the wedge sum of $|S| - 1$ copies of the circle $S^1$, so in particular
	$$\pi_1(X \backslash S) = \, \stackbin[{i=1}]{|S| -1} {\bigstar[-2.5pt]}\mathbb{Z}$$	is the free group on $|S| - 1$ generators. A representation of a free product on a set of groups $\{G_i\}_{i \in \mathcal{I}}$ consists of a vector space $V$ 
	and a set of actions $\{G_i \to \operatorname{GL}(V)\}_{i \in \mathcal{I}}$ which are independent of one another. Consequently, the data of a finite-dimensional $\pi_1(X \backslash S)$-representation is equivalent to a choice of positive integer $n$ and invertible linear transformations $T_i \in \operatorname{GL}(\mathbb{C},n)$ for each positive integer $i \leq |S|-1$. This gives rise to an $\mathcal{O}_{X\setminus S}$-coherent $\mathcal{D}_{X\setminus S}$-module, or equivalently, a locally free $\mathcal{O}_{X\setminus S}$-module with a flat connection $\nabla$.  If $S$ is nonempty, this is necessarily a free module (a trivial vector bundle); note that if $S$ is empty then we have a local system on all of $X=\mathbb{P}^1$ which is necessarily trivial. So in either case the representation gives rise to a flat connection on $\mathcal{O}_{X \setminus S}^n$. In a punctured complex disk about each point $s_j$, $\nabla$ must be gauge equivalent (via a holomorphic gauge equivalence) to $d-\frac{A_j}{z-s_j}$, for some matrix $A_j$ such that $T_j$ is conjugate to $\exp(2\pi i \cdot A_j)$.
	It is tempting then to try to write 
	\begin{equation}\label{connection_sum}
		\nabla = d - \sum_{i=1}^{|S| -1} \frac{A_i}{z - s_i} - p(z)dz
	\end{equation} for some (matrix-valued) polynomial $p(z)$. 
	However, note that \eqref{connection_sum} need not actually be gauge equivalent to $d-\frac{A_j}{z-s_j}$ in a (punctured) complex disc about $s_j$, even if $p(z)=0$:  in fact $\exp(2 \pi i A_j)$ need only be in the \emph{closure} of the conjugacy class of the monodromy of \eqref{connection_sum} about $s_j$. This can be seen by taking a limit of contour integrals which begin with a base point of the fundamental group, take a path to a point very close to $s_j$, take a small loop counterclockwise, and then take the inverse of the first path, as the small loop shrinks to $s_j$. 	(We remark that none of these issues arise in the case $n=1$: here, the monodromy of (5.1) about $s_j$ really is $\exp(2 \pi i A_j)$, for all $j$.)
	Nonetheless, it follows from the Riemann--Hilbert theorem that, for arbitrary $n$, every monodromy representation $(T_1,\ldots, T_{|S|-1})$ can be realised as a connection of the form \eqref{connection_sum} with a regular singularity at infinity, for some choices of the $A_j$ and $p(z)$ (with $\exp(2 \pi i A_j)$ in the closure of the conjugacy class of $T_j$ for all $j$). 
	
	Note that due to the presence of the $p(z)$ term, connections of the form \eqref{connection_sum} need not have a regular singularity at infinity, in general.  It is therefore natural to restrict to those having the property $p(z)=0$, which has a simple pole at infinity. These connections are called \emph{Fuchsian}.
	Hilbert's 21st problem asked whether every connection is gauge equivalent to a Fuchsian one; its answer was eventually found to be negative\footnote{A counterexample was first constructed in \cite{Bolibrukh_90}. In  \cite{Kostov}, it is shown that, assuming $|S| \geq 4$ and either $n = 3$ or $n \geq 7$, the codimension within $\operatorname{GL}(\mathbb{C},n)^{|S|-1}$ of the variety of monodromy groups which cannot be realised by a Fuchsian connection is $2(n-1)(|S|-1)$.}. However, we shall now restrict our attention to Fuchsian connections to avoid the possibility of an irregular singularity at $\infty$. If we are interested in irreducible connections---needed to form irreducible $\D$-modules and hence irreducible $\slt$-representations---we can make such an assumption without loss of generality: that irreducible connections on $X = \mathbb{CP}^1$ are always Fuchsian has been shown in \cite{Kostov} (Theorem 1) and independently in \cite{Bolibrukh_92} (Theorem 1).

	The connection $\nabla$ has a singularity at a point $s_j \in S$ iff there is \textit{no} local trivialisation of $\nabla$ in a punctured complex disc $U$ containing $s_j$ which is holomorphic at $s_j$. Different trivialisations are related to one another by gauge transformations, i.e. conjugation by $g$ for some  $g: U\setminus{s_j} \to \GL(\mathbb{C},n)$, 
  so it could happen that, even if $A_j$ is nonzero, $\nabla|_{U\setminus{s_j}}$ turns out to be holomorphic after it is conjugated by some such $g$.  (Thanks to the Riemann--Hilbert theorem, this is equivalent to $\nabla|_{X \setminus S}$ being gauge equivalent on $X\setminus S$ to a connection with no pole at $s_j$.)  Note that, for $\nabla$ of the form \eqref{connection_sum}, this could only happen if $A_j$ is a diagonalisable matrix with integral eigenvalues, i.e.,  $\exp(2 \pi i A_j) = I$. However, even in the latter case, it could still happen that there is a singularity at $s_j$ with unipotent monodromy (so that the identity is in the closure of its conjugacy class).  In the case $n=1$, this is a biconditional (any unipotent one-by-one matrix is the identity).
	
	Let us compute a concrete example of such a representation. We will consider the “simplest" case of three regular singularities; as we shall see, already this case harbours quite a degree of complexity. Write $S = \{0, 1, \infty\}$ where $0 = [0:1]$, $\infty = [1:0]$ and $1 = [1:1]$ and $U = X \backslash S$. Denote by $j : U \hookrightarrow X$ the open inclusion. Assume for the sake of simplicity that we are dealing with the trivial twist $\lambda = 1$ and choose two invertible linear transformations of dimension 1, i.e. non-zero complex scalars $\alpha, \beta \in \mathbb{C}^\times$. Define the $\Ox$-coherent $\Dx$-module $\M_{\alpha, \beta}= j_+(\Oc_U)$ with connection \begin{equation*}
		\nabla_{\alpha, \beta} = d - \frac{\alpha}{z} - \frac{\beta}{z-1} 
	\end{equation*} whose regular singularities lie within $S$. Specifically, $\M_{\alpha, \beta}$ has (regular) singularities at the points $0$ and $\infty$ if and only if $\alpha \notin \mathbb{Z}$ and it has a (regular) singularity at the point $1$ if and only if $\beta \notin \mathbb{Z}$.
	
	Since the support of $\M_{\alpha, \beta}$ is contained entirely in $U_0$ (in fact, it is contained entirely in $\ins$, so we could use either chart at will), we find that $$\G(X, \M_{\alpha, \beta}) \cong \G(U_0, \M_{\alpha, \beta})$$ so the $\slt$-representation corresponding to $\M_{\alpha , \beta}$ via Beilinson--Bernstein localisation is given by the local sections $\G(U_0, \M_{\alpha, \beta})$ together under the action specified in (\ref{infinitesimal}). We illustrate this representation in Figures \ref{fig:large1}--\ref{fig:large4}; in the interest of clarity, we split this figure into four sub-figures---one for the action of each of the three basis elements $E,F,H$, and one for the action of the element $A_1 = \frac{1}{2}H + F$, whose significance we explain below. To facilitate
	these arduous computations, we have devised a computer program for the computation of the representations that arise as local sections of $\Oc$-coherent $\D$-modules. We invite the reader to the appendix for an exposition of this program.
	
	Let us remark a few features of this representation. Consider the Cauchy-Euler operators $(z-y)\p_z$ and $(w-x) \p_w$ on the charts $U_0$ and $U_\infty$, respectively, centred at a point $p$ where $p = [1:y]$ or $p = [x:1]$, respectively. We can write the image of this operator under the Beilinson--Bernstein map $\Phi_1$ in terms of our fixed $\slt$-basis. Equation (\ref{infinitesimal}) yields $$A_p = z \p^z = \tfrac{1}{2}H_0 - y\cdot F_0$$ on the chart $U_0$, while equation (\ref{infinity}) yields $$A_p = w\p^w = -\tfrac{1}{2}H_\infty + x \cdot E_\infty $$ on the chart $U_\infty$. The action of this operator on the $\g$-representation $\G(X, \M)$ corresponding to a $\D$-module $\M$ is particularly interesting in the case where $p$ is a singularity of $\M$; $A_p$ then corresponds to a shift of the grading (by degree), by the order of the singularity at $p$. Indeed, if $\M$ has a pole of order $m$ at $p$ then \begin{equation}
		A_p \cdot \G(X, \Ox) \subseteq z^{m-1}\G(X, \Ox) \label{Euler}
	\end{equation} holds. In particular, $p$ is a singularity iff $A_p$ does not increase the degree of any regular function, and the singularity is regular iff $A_p$ preserves the degree of regular functions.
	
	In the case of the representation in Figures \ref{fig:large1}--\ref{fig:large4}, we are predominately interested in the Cauchy-Euler operators $$A_0 = A_\infty = \tfrac{1}{2}H \qquad \text{and} \qquad A_1 = \tfrac{1}{2}H + F$$ since the $\Dx$-module $\M_{\alpha, \beta}$ has singularities at the points $0 = [0:1], \infty = [1:0]$ and $1 = [1:1]$. The action of $A_0 \sim H$ is seen in Figure \ref{fig:large3}. Observe that the submodule generated by the basis element $\frac{1}{z-1}$ is connected to the remainder of the representation by a series of actions with multiplicities $-2(\beta + n)$ for $n \in \mathbb{Z}_{\leq 0}$. This corresponds to the fact that for integer $\beta$, the singularity at $1$ disappears. We may say that this operator “selects" the point at $1$: $\M_{\alpha , \beta}$ has a singularity there iff the $\langle A_0 \rangle$-module is the direct sum of two submodules, one of which is entirely supported at the point $1 = [1:1]$.
	
	Similarly, the action of $A_1$ is shown in Figure \ref{fig:large4}. Here we note a similar series connecting the basis elements of the form $z^n$ for $n \in \mathbb{Z}$, with multiplicities $n - \alpha$. Again, one of these is zero if and only if $\alpha$ is an integer, which in turn corresponds to the singularity at $z = 0$ or $w = 0$ disappearing. Thus $A_1$ “selects" the points $0$ and $\infty$, and $\M_{\alpha, \beta}$ is non-singular at these points iff the $\langle A_1 \rangle$-module may be written as a direct sum of two submodules, one of which is entirely supported on $\{0, \infty\}$.
	
	It is not a coincidence that the Cauchy-Euler operator at a point corresponds to the existence of a singularity of $\M$ at that point. The notion of regular singularities is defined locally in terms of the stability of $\M$ with respect to these operators (cf. Section 1). Perhaps this observation can be generalised to all holonomic $\D$-modules, which would yield a representation-theoretic analogue of the notion of regular singularities for $\D$-modules. Such a property would enable the combination of Beilinson--Bernstein localisation and the Riemann--Hilbert correspondence, and thus, open the door for a general topological theory of $\g$-representations superseding the examples outlined in this text. 

	\begin{figure}[ht]
	\centering
	\[
	\adjustbox{angle = 90, height = 0.9\textheight}{
			\begin{tikzcd}[ampersand replacement = \&]
				... \arrow[Red,rr, "\alpha + 4"] \arrow[Red,rrrrrrrrdd, "-\beta" description, bend right] \arrow[Red,rrrr, "\beta" description, bend left=49] \arrow[Red,rrrrrr, "\beta", bend left=49] \&  \& z^{-3} \arrow[Red,rr, "-(\alpha+3)"] \arrow[Red,rrrrrrdd, "-\beta" description, bend right] \arrow[Red,rrrr, "\beta" description, bend left=49] \&  \& z^{-2} \arrow[Red,rr, "-(\alpha+2)"] \arrow[Red,rrrrdd, "-\beta" description, bend right] \&  \& z^{-1} \arrow[Red,rrdd, "-\beta", bend right, shift left] \arrow[Red,rr, "-(\alpha+\beta+1)"] \&  \& 1 \arrow[Red,dd, "\beta"] \arrow[Red,rr, "-(\alpha + \beta)"]                                                                                                                             \&  \& z \arrow[Red,ll, "-\beta", shift left] \arrow[Red,lldd, "-\beta"] \arrow[Red,"-\beta"', loop, distance=2em, in=305, out=235] \arrow[Red,rr, "-(\alpha + \beta -1)"] \&  \& z^2 \arrow[Red,ll, "-\beta", shift left] \arrow[Red,lllldd, "-\beta", bend left] \arrow[Red,"-\beta"', loop, distance=2em, in=305, out=235] \arrow[Red,rr, "-(\alpha + \beta -2)"] \arrow[Red,llll, "-\beta" description, bend right=49] \&  \& z^3 \arrow[Red,ll, "-\beta", shift left] \arrow[Red,lllllldd, "-\beta", bend left] \arrow[Red,"-\beta"', loop, distance=2em, in=305, out=235] \arrow[Red,rr, "-(\alpha + \beta -3)"] \arrow[Red,llll, "-\beta" description, bend right=49] \arrow[Red,llllll, "-\beta" description, bend right=49] \&  \& ... \arrow[Red,ll, "-\beta", shift left] \arrow[Red,lllllllldd, "-\beta", bend left] \arrow[Red,llll, "-\beta" description, bend right=49] \arrow[Red,llllll, "-\beta" description, bend right=49] \arrow[Red,llllllll, "-\beta"', bend right=49] \\
				\&  \&                                                                                                                                      \&  \&                                                                                    \&  \&                                                                                        \&  \&                                                                                                                                                                                   \&  \&                                                                                                                                                      \&  \&                                                                                                                                                                                                                       \&  \&                                                                                                                                                                                                                                                                             \&  \&                                                                                                                                                                                                                                \\
				\&  \&                                                                                                                                      \&  \&                                                                                    \&  \&                                                                                        \&  \& (z-1)^{-1} \arrow[Red,dd, "-(\beta+1)", shift left] \arrow[Red,"-(\alpha + 2\beta + 2)" description, loop, distance=2em, in=305, out=235] \arrow[Red,uu, "-(\alpha + \beta + 1)", shift left] \&  \&                                                                                                                                                      \&  \&                                                                                                                                                                                                                       \&  \&                                                                                                                                                                                                                                                                             \&  \&                                                                                                                                                                                                                                \\
				\&  \&                                                                                                                                      \&  \&                                                                                    \&  \&                                                                                        \&  \&                                                                                                                                                                                   \&  \&                                                                                                                                                      \&  \&                                                                                                                                                                                                                       \&  \&                                                                                                                                                                                                                                                                             \&  \&                                                                                                                                                                                                                                \\
				\&  \&                                                                                                                                      \&  \&                                                                                    \&  \&                                                                                        \&  \& (z-1)^{-2} \arrow[Red,dd, "-(\beta+2)", shift left] \arrow[Red,"-(\alpha + 2\beta + 4)"', loop, distance=2em, in=215, out=145] \arrow[Red,uu, "-(\alpha + \beta + 2)", shift left]            \&  \&                                                                                                                                                      \&  \&                                                                                                                                                                                                                       \&  \&                                                                                                                                                                                                                                                                             \&  \&                                                                                                                                                                                                                                \\
				\&  \&                                                                                                                                      \&  \&                                                                                    \&  \&                                                                                        \&  \&                                                                                                                                                                                   \&  \&                                                                                                                                                      \&  \&                                                                                                                                                                                                                       \&  \&                                                                                                                                                                                                                                                                             \&  \&                                                                                                                                                                                                                                \\
				\&  \&                                                                                                                                      \&  \&                                                                                    \&  \&                                                                                        \&  \& (z-1)^{-3} \arrow[Red,dd, "-(\beta+3)", shift left] \arrow[Red,"-(\alpha + 2\beta + 6)"', loop, distance=2em, in=215, out=145] \arrow[Red,uu, "-(\alpha + \beta + 3)", shift left]            \&  \&                                                                                                                                                      \&  \&                                                                                                                                                                                                                       \&  \&                                                                                                                                                                                                                                                                             \&  \&                                                                                                                                                                                                                                \\
				\&  \&                                                                                                                                      \&  \&                                                                                    \&  \&                                                                                        \&  \&                                                                                                                                                                                   \&  \&                                                                                                                                                      \&  \&                                                                                                                                                                                                                       \&  \&                                                                                                                                                                                                                                                                             \&  \&                                                                                                                                                                                                                                \\
				\&  \&                                                                                                                                      \&  \&                                                                                    \&  \&                                                                                        \&  \& \ldots \arrow[Red,uu, "-(\alpha + \beta + 4)", shift left]                                                                                                                            \&  \&                                                                                                                                                      \&  \&                                                                                                                                                                                                                       \&  \&                                                                                                                                                                                                                                                                             \&  \&                                                                                                                                                                                                                               
			\end{tikzcd}                                                                                                 
	}
	\]
	\caption{Action of the basis element $\textcolor{Red}{E} \in \mathfrak{sl}_2(\mathbb{C})$ on a fragment of the basis for the representation $\G(X,\M_{\alpha, \beta}) = \G(X, j_+(\Oc_U))$.}
	\label{fig:large1}
\end{figure}
\begin{figure}[ht]
	\centering
	\[
	\adjustbox{angle = 90, height = 0.9\textheight}{
		\begin{tikzcd}[ampersand replacement = \&]
			... \arrow[Green,rrrrrrrrdd, "\beta" description, bend right] \arrow[Green,rrrr, "-\beta" description, bend left=49] \arrow[Green,rrrrrr, "-\beta", bend left=49] \&  \& z^{-3} \arrow[Green,ll, "\alpha + 3"'] \arrow[Green,"-\beta"', loop, distance=2em, in=305, out=235] \arrow[Green,rrrrrrdd, "\beta" description, bend right] \arrow[Green,rrrr, "-\beta" description, bend left=49] \&  \& z^{-2} \arrow[Green,ll, "\alpha + 2"'] \arrow[Green,"-\beta"', loop, distance=2em, in=305, out=235] \arrow[Green,rrrrdd, "\beta" description, bend right] \&  \& z^{-1} \arrow[Green,ll, "\alpha + 1"'] \arrow[Green,"-\beta"', loop, distance=2em, in=125, out=55] \arrow[Green,rrdd, "\beta", shift left] \&  \& 1 \arrow[Green,dd, "\beta" description] \arrow[Green,ll, "\alpha"']                                                                                                                                                                                    \&  \& z \arrow[Green,ll, "\alpha + \beta -1"'] \arrow[Green,lldd, "\beta"] \&  \& z^2 \arrow[Green,ll, "\alpha + \beta -2"'] \arrow[Green,lllldd, "\beta", bend left] \arrow[Green,llll, "\beta" description, bend right=49] \&  \& z^3 \arrow[Green,ll, "\alpha + \beta -3"'] \arrow[Green,lllllldd, "\beta", bend left] \arrow[Green,llll, "\beta" description, bend right=49] \arrow[Green,llllll, "\beta" description, bend right=49] \&  \& ... \arrow[Green,ll, "\alpha + \beta -4"'] \arrow[Green,lllllllldd, "\beta", bend left] \arrow[Green,llll, "\beta" description, bend right=49] \arrow[Green,llllll, "\beta" description, bend right=49] \arrow[Green,llllllll, "\beta"', bend right=49] \\
			\&  \&                                                                                                                                                                                             \&  \&                                                                                                                                           \&  \&                                                                                                                            \&  \&                                                                                                                                                                                                                                            \&  \&                                                            \&  \&                                                                                                                              \&  \&                                                                                                                                                                                     \&  \&                                                                                                                                                                                                                                   \\
			\&  \&                                                                                                                                                                                             \&  \&                                                                                                                                           \&  \&                                                                                                                            \&  \& (z-1)^{-1} \arrow[Green,dd, "\beta+1", shift left] \arrow[Green,"\alpha" description, loop, distance=2em, in=305, out=235] \arrow[Green,lluu, "-\alpha", shift left]                                                                                         \&  \&                                                            \&  \&                                                                                                                              \&  \&                                                                                                                                                                                     \&  \&                                                                                                                                                                                                                                   \\
			\&  \&                                                                                                                                                                                             \&  \&                                                                                                                                           \&  \&                                                                                                                            \&  \&                                                                                                                                                                                                                                            \&  \&                                                            \&  \&                                                                                                                              \&  \&                                                                                                                                                                                     \&  \&                                                                                                                                                                                                                                   \\
			\&  \&                                                                                                                                                                                             \&  \&                                                                                                                                           \&  \&                                                                                                                            \&  \& (z-1)^{-2} \arrow[Green,dd, "\beta+2", shift left] \arrow[Green,"\alpha"', loop, distance=2em, in=215, out=145] \arrow[Green,lluuuu, "-\alpha" description, bend left] \arrow[Green,uu, "-\alpha", shift left]                                                     \&  \&                                                            \&  \&                                                                                                                              \&  \&                                                                                                                                                                                     \&  \&                                                                                                                                                                                                                                   \\
			\&  \&                                                                                                                                                                                             \&  \&                                                                                                                                           \&  \&                                                                                                                            \&  \&                                                                                                                                                                                                                                            \&  \&                                                            \&  \&                                                                                                                              \&  \&                                                                                                                                                                                     \&  \&                                                                                                                                                                                                                                   \\
			\&  \&                                                                                                                                                                                             \&  \&                                                                                                                                           \&  \&                                                                                                                            \&  \& (z-1)^{-3} \arrow[Green,dd, "\beta+3", shift left] \arrow[Green,"\alpha"', loop, distance=2em, in=215, out=145] \arrow[Green,lluuuuuu, "-\alpha" description, bend left] \arrow[Green,uu, "-\alpha", shift left] \arrow[Green,uuuu, "\alpha" description, bend right=49] \&  \&                                                            \&  \&                                                                                                                              \&  \&                                                                                                                                                                                     \&  \&                                                                                                                                                                                                                                   \\
			\&  \&                                                                                                                                                                                             \&  \&                                                                                                                                           \&  \&                                                                                                                            \&  \&                                                                                                                                                                                                                                            \&  \&                                                            \&  \&                                                                                                                              \&  \&                                                                                                                                                                                     \&  \&                                                                                                                                                                                                                                   \\
			\&  \&                                                                                                                                                                                             \&  \&                                                                                                                                           \&  \&                                                                                                                            \&  \& \ldots \arrow[Green,lluuuuuuuu, "-\alpha" description, bend left] \arrow[Green,uu, "-\alpha", shift left] \arrow[Green,uuuu, "\alpha" description, bend right=49] \arrow[Green,uuuuuu, "-\alpha"', bend right=49]                                                  \&  \&                                                            \&  \&                                                                                                                              \&  \&                                                                                                                                                                                     \&  \&                                                                                                                                                                                                                                  
		\end{tikzcd}
	}
	\]
	\caption{Action of the basis element $\textcolor{Green}{F} \in \mathfrak{sl}_2(\mathbb{C})$ on a fragment of the basis for the representation $\G(X,\M_{\alpha, \beta}) = \G(X, j_+(\Oc_U))$.}
    \label{fig:large2}
\end{figure}

\begin{figure}[ht]
	\centering
	\[
	\adjustbox{angle = 90, height = 0.9\textheight}{
	\begin{tikzcd}[ampersand replacement = \&]
		... \arrow[Cyan,rr, "2\beta"] \arrow[Cyan,rrrrrrrrdd, "-2\beta"', bend right] \arrow[Cyan,rrrr, "-2\beta" description, bend left=49] \arrow[Cyan,rrrrrr, "-2\beta", bend left=49] \&  \& z^{-3} \arrow[Cyan,rr, "2\beta"] \arrow[Cyan,"-2(\alpha + 3)"', loop, distance=2em, in=305, out=235] \arrow[Cyan,rrrrrrdd, "-2\beta"', bend right] \arrow[Cyan,rrrr, "-2\beta" description, bend left=49] \&  \& z^{-2} \arrow[Cyan,rr, "2\beta"] \arrow[Cyan,"-2(\alpha + 2)"', loop, distance=2em, in=305, out=235] \arrow[Cyan,rrrrdd, "-2\beta"', bend right] \&  \& z^{-1} \arrow[Cyan,"-2(\alpha + 1)"', loop, distance=2em, in=305, out=235] \arrow[Cyan,rrdd, "-2\beta"'] \&  \& 1 \arrow[Cyan,dd, "-2\beta" description] \arrow[Cyan,"-2(\alpha+\beta)"', loop, distance=2em, in=125, out=55]          \&  \& z \arrow[Cyan,"-2(\alpha+\beta -1)"', loop, distance=2em, in=305, out=235] \arrow[Cyan,ll, "-2\beta"'] \arrow[Cyan,lldd, "-2\beta"] \&  \& z^2 \arrow[Cyan,"-2(\alpha + \beta -2)"', loop, distance=2em, in=305, out=235] \arrow[Cyan,ll, "-2\beta"'] \arrow[Cyan,lllldd, "-2\beta", bend left] \arrow[Cyan,llll, "-2\beta" description, bend right=49] \&  \& z^3 \arrow[Cyan,"-2(\alpha + \beta -3)"', loop, distance=2em, in=305, out=235] \arrow[Cyan,ll, "-2\beta"'] \arrow[Cyan,lllllldd, "-2\beta", bend left] \arrow[Cyan,llll, "-2\beta" description, bend right=49] \arrow[Cyan,llllll, "-2\beta" description, bend right=49] \&  \& ... \arrow[Cyan,ll, "-2\beta"'] \arrow[Cyan,lllllllldd, "-2\beta", bend left] \arrow[Cyan,llll, "-2\beta" description, bend right=49] \arrow[Cyan,llllll, "-2\beta" description, bend right=49] \arrow[Cyan,llllllll, "-2\beta"', bend right=49] \\
		\&  \&                                                                                                                                                                                       \&  \&                                                                                                                                   \&  \&                                                                                                \&  \&                                                                                                              \&  \&                                                                                                                      \&  \&                                                                                                                                                                                          \&  \&                                                                                                                                                                                                                                                 \&  \&                                                                                                                                                                                                                         \\
		\&  \&                                                                                                                                                                                       \&  \&                                                                                                                                   \&  \&                                                                                                \&  \& (z-1)^{-1} \arrow[Cyan,dd, "-2(\beta + 1)"] \arrow[Cyan,loop, distance=2em, in=215, out=145]                           \&  \&                                                                                                                      \&  \&                                                                                                                                                                                          \&  \&                                                                                                                                                                                                                                                 \&  \&                                                                                                                                                                                                                         \\
		\&  \&                                                                                                                                                                                       \&  \&                                                                                                                                   \&  \&                                                                                                \&  \&                                                                                                              \&  \&                                                                                                                      \&  \&                                                                                                                                                                                          \&  \&                                                                                                                                                                                                                                                 \&  \&                                                                                                                                                                                                                         \\
		\&  \&                                                                                                                                                                                       \&  \&                                                                                                                                   \&  \&                                                                                                \&  \& (z-1)^{-2} \arrow[Cyan,dd, "-2(\beta + 2)"] \arrow[Cyan,"2(\alpha + \beta + 2)"', loop, distance=2em, in=215, out=145] \&  \&                                                                                                                      \&  \&                                                                                                                                                                                          \&  \&                                                                                                                                                                                                                                                 \&  \&                                                                                                                                                                                                                         \\
		\&  \&                                                                                                                                                                                       \&  \&                                                                                                                                   \&  \&                                                                                                \&  \&                                                                                                              \&  \&                                                                                                                      \&  \&                                                                                                                                                                                          \&  \&                                                                                                                                                                                                                                                 \&  \&                                                                                                                                                                                                                         \\
		\&  \&                                                                                                                                                                                       \&  \&                                                                                                                                   \&  \&                                                                                                \&  \& (z-1)^{-3} \arrow[Cyan,dd, "-2(\beta + 3)"] \arrow[Cyan,"2(\alpha + \beta + 3)"', loop, distance=2em, in=215, out=145] \&  \&                                                                                                                      \&  \&                                                                                                                                                                                          \&  \&                                                                                                                                                                                                                                                 \&  \&                                                                                                                                                                                                                         \\
		\&  \&                                                                                                                                                                                       \&  \&                                                                                                                                   \&  \&                                                                                                \&  \&                                                                                                              \&  \&                                                                                                                      \&  \&                                                                                                                                                                                          \&  \&                                                                                                                                                                                                                                                 \&  \&                                                                                                                                                                                                                         \\
		\&  \&                                                                                                                                                                                       \&  \&                                                                                                                                   \&  \&                                                                                                \&  \& \ldots                                                                                                       \&  \&                                                                                                                      \&  \&                                                                                                                                                                                          \&  \&                                                                                                                                                                                                                                                 \&  \&                                                                                                                                                                                                                        
	\end{tikzcd}
	}
	\]
	\caption{Action of the basis element $\textcolor{Cyan}{H} \in \mathfrak{sl}_2(\mathbb{C})$ on a fragment of the basis for the representation $\G(X,\M_{\alpha, \beta}) = \G(X, j_+(\Oc_U))$.}
	\label{fig:large3}
\end{figure}
\begin{figure}[ht]
	\centering
	\[
	\adjustbox{angle = 90, height = 0.9\textheight}{
	\begin{tikzcd}[ampersand replacement = \&]
		... \&  \& z^{-3} \arrow[ll, "\alpha + 3"'] \arrow["-(\alpha + \beta + 3)"', loop, distance=2em, in=125, out=55] \&  \& z^{-2} \arrow[ll, "\alpha + 2"'] \arrow["-(\alpha + \beta + 2)"', loop, distance=2em, in=125, out=55] \&  \& z^{-1} \arrow[ll, "\alpha + 1"'] \arrow["-(\alpha + \beta + 1)"', loop, distance=2em, in=125, out=55] \&  \& 1 \arrow[ll, "\alpha"'] \arrow["-(\alpha + \beta)"', loop, distance=2em, in=125, out=55]                                                                                                \&  \& z \arrow[ll, "\alpha - 1"', shift left] \arrow["-(\alpha + \beta -1)"', loop, distance=2em, in=125, out=55] \&  \& z^2 \arrow[ll, "\alpha - 2"', shift left] \arrow["-(\alpha + \beta -2)"', loop, distance=2em, in=125, out=55] \&  \& z^3 \arrow[ll, "\alpha - 3"', shift left] \arrow["-(\alpha + \beta -3)"', loop, distance=2em, in=125, out=55] \&  \& ... \arrow[ll, "\alpha - 4"', shift left] \\
		\&  \&                                                                                                       \&  \&                                                                                                       \&  \&                                                                                                       \&  \&                                                                                                                                                                                         \&  \&                                                                                                             \&  \&                                                                                                               \&  \&                                                                                                               \&  \&                                           \\
		\&  \&                                                                                                       \&  \&                                                                                                       \&  \&                                                                                                       \&  \& (z-1)^{-1} \arrow["-(\beta + 1)"', loop, distance=2em, in=125, out=55] \arrow[lluu, "-\alpha"', bend left]                                                                              \&  \&                                                                                                             \&  \&                                                                                                               \&  \&                                                                                                               \&  \&                                           \\
		\&  \&                                                                                                       \&  \&                                                                                                       \&  \&                                                                                                       \&  \&                                                                                                                                                                                         \&  \&                                                                                                             \&  \&                                                                                                               \&  \&                                                                                                               \&  \&                                           \\
		\&  \&                                                                                                       \&  \&                                                                                                       \&  \&                                                                                                       \&  \& (z-1)^{-2} \arrow["-(\beta + 2)"', loop, distance=2em, in=215, out=145] \arrow[uu, "-\alpha", shift left] \arrow[lluuuu, "\alpha", bend left]                                           \&  \&                                                                                                             \&  \&                                                                                                               \&  \&                                                                                                               \&  \&                                           \\
		\&  \&                                                                                                       \&  \&                                                                                                       \&  \&                                                                                                       \&  \&                                                                                                                                                                                         \&  \&                                                                                                             \&  \&                                                                                                               \&  \&                                                                                                               \&  \&                                           \\
		\&  \&                                                                                                       \&  \&                                                                                                       \&  \&                                                                                                       \&  \& (z-1)^{-3} \arrow["-(\beta + 3)"', loop, distance=2em, in=215, out=145] \arrow[uu, "-\alpha", shift left] \arrow[lluuuuuu, "-\alpha", bend left] \arrow[uuuu, "\alpha"', bend right=49] \&  \&                                                                                                             \&  \&                                                                                                               \&  \&                                                                                                               \&  \&                                           \\
		\&  \&                                                                                                       \&  \&                                                                                                       \&  \&                                                                                                       \&  \&                                                                                                                                                                                         \&  \&                                                                                                             \&  \&                                                                                                               \&  \&                                                                                                               \&  \&                                           \\
		\&  \&                                                                                                       \&  \&                                                                                                       \&  \&                                                                                                       \&  \& \ldots \arrow[uu, "-\alpha", shift left] \arrow[lluuuuuuuu, "\alpha", bend left] \arrow[uuuu, "\alpha"', bend right=49]                                                                 \&  \&                                                                                                             \&  \&                                                                                                               \&  \&                                                                                                               \&  \&                                          
	\end{tikzcd}
	}
	\]
	\caption{Action of the element $A_1 = \frac{1}{2}\textcolor{Cyan}H + \textcolor{Green}{F} \in \mathfrak{sl}_2(\mathbb{C})$ on a fragment of the basis for the rep. $\G(X,\M_{\alpha, \beta}) = \G(X, j_+(\Oc_U))$.}
	\label{fig:large4}
\end{figure}
	
	\section*{Appendix: Computation of Beilinson--Bernstein localisation}
	Calculating the representations which correspond to an $\Oll$-coherent $\Dl$-module $\M$ on $X = \mathbb{CP}^1$ can be a laborious task, especially in the case where $\M$ has more than two singularities. For this reason, we created a python script which does the necessary calculations. We invite the reader to download it from: 
	\begin{center}
	    \url{https://github.com/Julek99/DModRep}
	\end{center}
	and we outline briefly in this appendix how to use it. The user has two programs at their disposal --- \texttt{compute.py} for computation and \texttt{graph.py} for graphical representation. A third script, \texttt{core.py}, is also included in the file; this script contains the actual mathematical machinery, while the others are merely user interfaces which make use of \texttt{core.py}. For this reason, it is essential that \texttt{core.py} is kept in the same folder as the other programs.
	
	Each program takes as input a connection $\nabla$ defined locally on one of the charts $U_0$ and $U_\infty$ which the user has specified. The user is then prompted to give a basis for $\M$, in terms of the coordinates $z$ on $U_0$ or $w$ on $U_\infty$. The script then calculates the action of $\slt$-elements on the local sections of $\M$ and generates LaTeX code for a graphical representation of this action.
	
	Note: To run properly, this script requires Python 3.7 and the library SymPy (cf. \cite{10.7717/peerj-cs.103}) installed.
	
	\subsubsection*{Computation}
	We illustrate the function of the program \texttt{compute.py} using the example of Whittaker modules (cf. Section 8 of \cite{romanov} or \cite{whittaker}). These are representations which correspond to the module $(j_\infty)_+(\Oc_{U_\infty})$ with connection twisted by some $\eta \in \mathbb{C}$, i.e. \begin{equation}
	   \nabla = d - \eta \tag{$\ast$} \label{eta}
	\end{equation} on $U_\infty$. As this module is fully supported in $U_\infty$, the global sections agree with the local sections on $U_\infty$, so we may use the program to compute the corresponding representation.
	
	To begin, run the python script \texttt{compute.py}, which can be done on a MacOS or Linux machine (with Python 3.X and SymPy installed) by typing $$\texttt{python3 compute.py}$$ in a terminal window which has been navigated to the folder where these scripts are saved. When prompted to specify a chart, type $$\texttt{inf}$$ to select the chart $U_\infty$ with coordinate mapping $w$. When prompted to define variables, type $$\texttt{e} \qquad \text{and} \qquad \backslash\texttt{eta}$$ for the variable name and variable description, respectively. This lets the program know that we will be using a variable $\eta$, and we may refer to it as $\texttt{e}$. When prompted to define a basis, type \texttt{w} which lets the program know that the basis we will be using is $\{w^n\}_{n = 0}^\infty$. After confirming that this is the only basis, we can input the connection. We type $\texttt{e}$ to input the connection (\ref{eta}). At this stage, the program initialises a representation ready for computations. To view the action of an element $A \in \slt$, we input $A$ as a combination of the basis elements $E,F$ and $H$. For instance, we can write $$\texttt{A = E*F + 2*H}$$ to find the action of $A = E_\infty F_\infty + 2H_\infty$. On $\texttt{n = 5}$ powers of our basis $\texttt{w}$, the program computes:
	\begin{align*}
	    A \cdot 1 &= 4 \eta w + \eta (-\eta w^2 + w (1 - \lambda)) + 2 \lambda - 2 \\
        A \cdot w^1 &= 4 \eta w^2 + w (2 \lambda - 6) + (\eta w - 1) (-\eta w^3 + w^2 (2 - \lambda)) \\
        A \cdot w^2 &= 4 \eta w^3 + w^2 (2 \lambda - 10) + (\eta w^2 - 2 w) (-\eta w^4 + w^3 (3 - \lambda))\\
        A \cdot w^3 &= 4 \eta w^4 + w^3 (2 \lambda - 14) + (\eta w^3 - 3 w^2) (-\eta w^5 + w^4 (4 - \lambda))\\
        A \cdot w^4 &= 4 \eta w^5 + w^4 (2 \lambda - 18) + (\eta w^4 - 4 w^3) (-\eta w^6 + w^5 (5 - \lambda))\\
        A \cdot w^5 &= 4 \eta w^6 + w^5 (2 \lambda - 22) + (\eta w^5 - 5 w^4) (-\eta w^7 + w^6 (6 - \lambda))
	\end{align*}
    which we invite the reader to compare with the result in Section 8 of \cite{romanov}. 

    \subsubsection*{Graphical Representation}
    Currently, graphical representation is only possible for modules with a basis generated (via taking powers) by a single element\footnote{For this reason, Figures \ref{fig:large1}--\ref{fig:large4} could not be generated using the script \texttt{graph.py}. Instead, we used the script \texttt{compute.py} to compute the $\slt$-action on this representation (which works perfectly fine on bases of this form) and then made the figure manually. Here we are indebted to the interface available at \url{https://github.com/yishn/tikzcd-editor}, which significantly facilitated the creation of such an intricate diagram.}, i.e. of the form $\{b^n\}_{n = 0}^\infty$ for some $b \in \M$. Tu illustrate the representation corresponding via Beilinson--Bernstein to an $\Oll$-coherent $\Dl$-module $\M$, run the script $\texttt{graph.py}$ and follow the same procedure as in case of computation above.
    
    The script returns the LaTeX code of a Tikz-CD figure which depicts the action of $A$ on $\G(U_i,\M)$ where $i \in \{0,\infty\}$ as chosen. In the case of $A = H$, this code compiles as Figure \ref{fig:script}, which depicts the action of this $\slt$-element on a Whittaker module. The Whittaker module corresponding to the connection (\ref{eta}) has a unique (irregular) singularity at the point $\infty$. Consequently, the Cauchy-Euler operator $A_\infty = \tfrac{1}{2}H$ at $\infty$ does not preserve the grading on this module by degree in the coordinate function $w$. 
    
    \begin{figure}
	\centering
    \[ \begin{tikzcd}
    1  \arrow[ loop above,"{\lambda - 1}"] \arrow[r,bend right = 0,"{2\eta}"] & 
    (w)^{1}  \arrow[ loop above,"{\lambda - 3}"] \arrow[r,bend right = 0,"{2\eta}"] & 
    (w)^{2}  \arrow[ loop above,"{\lambda - 5}"] \arrow[r,bend right = 0,"{2\eta}"] & 
    (w)^{3}  \arrow[ loop above,"{\lambda - 7}"] \arrow[r,bend right = 0,"{2\eta}"] & 
    (w)^{4}  \arrow[ loop above,"{\lambda - 9}"]
    \end{tikzcd}\]
	\caption{The action of $A = H$ on the Whittaker module corresponding to the connection $\nabla = d - \eta$. Figure created automatically using Beilinson--Bernstein localisation calculator.}
	\label{fig:script}
\end{figure}

	\section*{Acknowledgements}
	This article grew out of a summer research project of the undergraduate research opportunities programme (UROP) organised at Imperial College London. We’d like to thank Gabriel Ng Chong Hui for his very valuable input throughout the project, as well as his lasting involvement in the creation of the article after the end of the project. We’d also like to thank Micah Gay for her collaboration undertaken during a predecessor project. Finally, we’d like to thank Prof. Pavel Etingof for having originally suggested the question behind this article, and for his very helpful corrections regarding Section 5.
	
	\sloppy
	\printbibliography
\end{document}